\documentclass[12pt]{amsart}     

\usepackage{amsfonts}
\usepackage{amsmath}
\usepackage{graphicx}
\usepackage{amsfonts} \usepackage{amsmath, amsthm}
\usepackage{epsfig} \usepackage{verbatim}
\usepackage{amssymb, amsthm, amsmath, bm, tikz, enumerate, color}
\usepackage{cite}

\newtheorem{theorem}{Theorem}

\newtheorem{corollary}[theorem]{Corollary}

\newtheorem{lemma}{Lemma}

\newtheorem{proposition}{Proposition}
\newtheorem{remark}[theorem]{Remark}

\newcommand{\vol}{\mathop{\mathrm{vol}}\nolimits}

\newcommand{\dist}{\mathop{\mathrm{dist}}\nolimits}

\newcommand{\length}{{\mathrm{length}}}
\newcommand{\Prob}{{\mathbb{P}}}
\newcommand{\cD}{{\mathcal{D}}}
\newcommand{\cF}{{\mathcal{F}}}
\newcommand{\cG}{{\mathcal{G}}}
\newcommand{\cS}{{\mathcal{S}}}
\newcommand{\eps}{{\varepsilon}}
\newcommand{\bc}{{\mathbf{c}}}

\newcommand{\hn}{{\hat n}}

\newcommand{\tH}{{\tilde h}}
\DeclareMathOperator{\sgn}{sgn}

\title[Non equilibrium density profiles in thermostated Lorentz tubes] 
{Non equilibrium density profiles in Lorentz tubes with thermostated boundaries}

\author{Dmitry Dolgopyat}
\address[Dmitry Dolgopyat]{University of Maryland}
\email{dmitry@math.umd.edu}

\author{P\'{e}ter N\'{a}ndori}
\address[P\'{e}ter N\'{a}ndori]{Courant Institute}
\email{nandori@cims.nyu.edu}

\begin{document}

\begin{abstract}
We consider a long Lorentz tube with absorbing boundaries.
Particles are injected to the tube from the left end. 
We compute the equilibrium density profiles in two cases: the semi-infinite tube 
(in which case the density is constant) and 
a long finite tube (in which case the density is linear).
In the latter case, we also show that convergence to equilibrium is well described by the heat equation.
In order to prove these results, we obtain new results for the Lorentz particle which are of independent interest.
First, we show that a particle conditioned not to hit the boundary for a long time 
converges to the Brownian meander. Second, we prove several local limit theorems for particles having a prescribed 
behavior in the past.
\end{abstract}

\maketitle

\section{Introduction}
An important problem in mathematical physics is to understand the emergence of macroscopic equations from
deterministic microscopic laws (see e.g. reviews \cite{BLRB00, Bu00, ChD06, LSp83, Sp80, Sp91, Sz00}). In particular, one would like to derive the
Fourier law for transport of conserved quantities. So far, this task has only been achieved for one 
deterministic system: Lorentz gas \cite{BBS83, BSC91, Ga69, Sp80}.
Even in that case our understanding is not complete.   
First, the Fourier law is derived for the ideal gas of non-interacting particles which is assumed to be at equilibrium.
However, the ideal gas can not reach the equilibrium since in the absence of interactions the energy of each
particle is conserved. Therefore, it is desirable to understand how the Lorentz gas achieves the equilibrium 
if the particles interact weakly with each other. Second, there are several ways to define the transport coefficients.
In particular, one can consider 

(i) particles in the whole space 

(ii) particles confined to vessel with impenetrable boundaries

(iii) particles in a certain region whose boundary is kept at a given temperature by means of a thermostat. 

For physicists, those definitions are clearly
equivalent but mathematically they are different. In particular, boundary layers need to be studied in the second
and third case. Case (i) has been analyzed in \cite{BSC91} for periodic Lorentz gas and in \cite{Ga69, Sp78, BBS83} for 
random Lorentz gas in Boltzmann-Grad limit. Case (ii)
has been studied in \cite{DSzV09} for periodic Lorentz gas and in \cite{LSp78} for 
random Lorentz gas in Boltzmann-Grad limit. The present paper deals with case (iii). 

We consider a strip on a plane with a periodic configuration of convex scatters removed.
We assume that the domain has finite horizon (that is, the particle can not move indefinitely without hitting a scatterer)
since an anomalous transport takes place in the infinite horizon case \cite{Bl92, SzV07, ChD09A, MS10}. Moving particles 
are injected from
the left end of the tube according to a Poisson process with constant intensity. We assume that the particles move with
the unit speed and that their initial position and direction are random. When the particle hits an end of the tube it disappears
from the system.
First, we consider a semi-infinite tube and 
show (Theorem \ref{thm:1}) that at equilibrium (that is, if we start injecting the particles at time $-\infty$)
the density of particles approaches a finite limit as the distance from the boundary tends to infinity.
The physical meaning of this result is that the particle density at the boundary is well defined.
Next, we show (Theorem \ref{thm:2}) that if we have have a large finite tube, then the equilibrium 
density profile is linear interpolating between the limiting densities at the end points (by the superposition
principle it suffices to consider the case where particles are injected only from the left). Finally, 
we show (Theorem \ref{thm:4}) that
that if we start from a non-equilibrium profile then the approach to equilibrium is described by the heat
equation. 

To derive Theorems \ref{thm:1}, \ref{thm:2} and \ref{thm:4}, we obtained several new results for 
one Lorentz particle. First, we show (Theorem \ref{thm:meander})
that a particle conditioned not to hit the boundary for a long time 
converges to the Brownian meander. Second, we prove several local limit theorems for particles having a prescribed 
behavior in the past (see Section \ref{SSLLT} for precise formulations). 
There are two novel features of our local limit theorems. First,
since our system has no translational symmetry (due to the presence of the boundaries)
we can not use Fourier analysis. Second, we are able to obtain local limit theorems 
conditioning on events of small probability in both past and future. These results seem to be of independent interest.
First, the fact that we can gain a very precise information about the distribution of the particle at a given time $t$
can be useful for studying weakly interacting particles. Secondly, local limit theorems have been used in \cite{DSzV08}
to compute the limiting distribution of ergodic averages for certain infinite measure preserving transformations related
to the Lorentz system and we can hope to get similar results for the semi-infinite tubes. Third, our result should be helpful
for analyzing Lorentz process with small deterministic holes (see \cite{NSz12} for the case of random holes).

The layout of our paper is the following. In Section \ref{ScPrel} 
we provide the necessary definitions and review the results from the 
theory of Sinai billiards which will be used in the sequel. Section \ref{ScRes} contains precise formulations
of our results. 
In Section \ref{ScHalfLine} we prove the equilibrium profile in a semi-infinite tube.
Section \ref{ScMeander} treats the convergence to Brownian meander.
Section \ref{ScLLT} contains the proofs of the new local limit results we need.
In Section \ref{ScFiniteTube}
we study the equilibrium profile in a long finite tube.
In Section \ref{ScHeat} we discuss the convergence to equilibrium.

The paper has two appendices. In Appendix \ref{AppLLT} we extend the usual Local Limit Theorem for Lorentz particle
to ensure the uniformity with respect to a large class of initial measures and also to provide the bound for cells
which are further from the origin than predicted by diffusive scaling. Appendix \ref{AppMeander} contains some computations
involving the density of the Brownian meander.

\section{Preliminaries}
\label{ScPrel}

\subsection{Notation.}
In this paper we denote every universal constant by $C$, thus each occurrence of $C$ may stand for
a different number. We also write $\mathbb P (A|B) =\mathbb P(A \cap B) / \mathbb P (B)$.

\subsection{Sinai billiard}

Here, we summarize briefly the most important notions from the theory of Sinai 
billiards needed in the present work. For a much ampler description, 
consult \cite{CM06}.
Define 
$\mathcal D = \mathbb R \times \mathcal S^1 \setminus \cup_{i=1}^{\infty} B_i$,
where $B_1, \dots, B_k$ are disjoint strictly convex domains inside the unit
torus, whose boundaries
are $C^3$-smooth and whose curvatures are bounded from below. 
$B_{k+1}, B_{k+2}, ...$ are the translational copies of $B_1, \dots, B_k$ 
with translations in $\mathbb Z $. 
The billiard flow is the dynamics of a point particle in $\mathcal D$, 
which consists of free flight inside $\mathcal D$ and specular reflection
on $\partial \mathcal D$. Since the speed is constant, is it assumed to be $1$.
Thus the billiard flow $\Phi^t$ acts on the space 
$\mathcal D \times \mathcal S^1$. For $(x_1, x_2) \in \mathcal D$, $v \in \mathcal S^1$,
and $\Phi^t ((x_1, x_2),v) = ((x_1', x_2'), v')$, we will write
$\hat X(t) = \hat X((x_1, x_2),v, t) = x_1'(t)$, for the horizontal component
of the position at (continuous) time $t$.

It is common to take the Poincar\'e section
on the boundaries of the scatterer, and switch to a discrete time dynamics, which is
called the billiard map.
The phase space of the billiard
map is
\[ \mathcal M = \{ x=(q,v) \in \partial \mathcal D \times S^1, 
\langle v,n \rangle \geq 0\},\]
where $n$ is the normal vector of $\partial D$ at the point $q$ pointing inside
$\mathcal D$, and the map itself is denoted by 
$\mathcal F: \mathcal M \rightarrow \mathcal M$. The natural invariant
measure on $\mathcal M$, denoted by $\mu$, is the projection
of the Lebesgue measure on the phase space of the billiard flow. 
In fact, $d \mu
= \mathrm{const}  \cos \phi dr d \phi$, where $r$ is the arc length parameter on 
$\partial \mathcal D$ and $\phi \in [- \pi /2, \pi /2]$ is the angle 
between $v$ and $n$. We will write $q(x)$ for the
the projection of the point $x$ to its first coordinate (that is
$q(x) \in \partial \mathcal D$).
The free flight vector $\kappa (x)$ is the lifted version of 
$q(\mathcal F (x)) - q(x)$ form $\mathbb R \times \mathcal S^1$ to 
$\mathbb R ^2$ (that would be the same as $q(\mathcal F (x)) - q(x)$ if the 
Lorentz process was defined in the plane, i.e. $B_{k+1}, B_{k+2}, ...$ where translational copies
of $B_1, ..., B_k$ with translations in $\mathbb Z^2$). We assume that $\kappa$
is bounded, i.e. $\kappa_{\min} \leq |\kappa| \leq \kappa_{\max}$ (the so-called
finite horizon condition), and write
\begin{equation}
\label{eq:xn}
 X_k = X_k(q,v) = \Pi \sum_{i=0}^{k-1} \kappa(\mathcal F^i (q,v)),
\end{equation}
where $\Pi$ is the projection to the horizontal direction (that is, 
$X_k$ is the discrete counterpart of $\hat X(t)$). We also denote by
\begin{equation}
\label{eq:freeflight}
 F_k = F_k(q,v) =  \sum_{i=0}^{k-1} | \kappa(\mathcal F^i (q,v))|,
\end{equation}
the time of the $k$-th collision.

Analogously, one can define the Sinai billiard on the torus $\mathbb T^2=
\mathbb R^2 / \mathbb Z ^2$. Then one needs to introduce 
$\mathcal D_0 = \mathbb T^2 \setminus \cup_{i=1}^{k} B_i$, and define
$\mathcal M_0$, $\mathcal F_0$ and $\mu_0$ as before. $\mu$ is the 
periodic extension of $\mu_0$. Since $\mu$ is infinite and $\mu_0$ is 
finite, we choose the constant in the definition of $\mu$ so that
$\mu_0$ is a probability measure. Finally, we write $\bar \kappa = \int |\kappa| d \mu_0$
for the mean free path length.

Since we are going to consider tubes with absorbing walls, hitting times are 
very important. Let ${\hat \tau}_L$ denotes the first time instant,
when the particle reaches the horizontal distance $L$, i.e.
${\hat \tau}_L = \inf \{ s>0: {\hat X}(s) =L \}$, and $\tau_L$ is its discrete counterpart, i.e.
$\tau_L = \min \{ k: \lfloor X_k \rfloor =L \}$.
We also write ${\hat \tau}^* = \hat \tau_0$ and $\tau^* = \tau_{-1}$ 
(this is the time of absorption in the case of semi infinite tube).\\
Hyperbolicity and ergodicity of $\mathcal F_0$ (nice properties) were 
proven by Sinai \cite{S70}. An unpleasant property of the billiard
map is the presence of singularities (corresponding to grazing collisions).
To overcome the technical difficulties caused by the singularities,
we use the so-called standard pair method developed in \cite{ChD09B}. 
Below we present an informal description of this method, see \cite{CM06} for more details.

For almost every $x \in \mathcal M_0$, stable and unstable manifolds through
$x$ exist. There is a factor of stretching in the unstable direction, which
is bounded from below by some $\Lambda >1$. Nevertheless, these factors are
not bounded from above (if $x$ is very close to a grazing collision where 
$\{ \cos \phi =0 \}$, the expansion is very big),
which makes it difficult to control the distortion
of unstable manifolds. That is why it is common to introduce the 
following additional (secondary)
singularities 
\[ S_{ \pm k } = \{ (r, \phi): \phi = \pm \pi/2 \mp k^{-2} \} \]
for $k$ larger than some $k_0$, yielding bounded distortion of
an unstable manifold disjoint to all singularities.
An unstable curve is some curve $W \subset \mathcal M_0$ 
such that at every point $x \in W$,
the tangent space $T_x W$ is in the unstable cone (slightly weaker
property than the unstable manifold). Further, $W$ is 
homogeneous, if does not intersect any singularity. A pair 
$\ell=(W, \rho)$ is called a standard pair, if $W$ is a homogeneous
unstable curve and $\rho$ is a regular probability measure supported
on $W$. Precisely, the regularity required for the measures is the following:
$$ \bigg| \log \frac{ d \rho}{d Leb}(x) - \log \frac{ d \rho}{d Leb}(y) \bigg| \leq C_0 \frac{|W(x,y)|}{|W|^{2/3}}, $$
where $C_0$ is a fixed constant and $|W(x,y)|$ is the arc length of the segment of $W$ lying between $x$ and $y$ 
(see \cite{ChD09B} for more details). In particular, the logarithm of the density of $\rho$ is uniformly H\"older 
continuous.
For a standard pair $\ell=(W, \rho)$, we write $\mathbb{E}_{\ell}$ for the integral
with respect to $\rho$, $\mathbb{P}_{\ell} (A) = \mathbb E_{\ell}({\bf 1}_A)$ and 
$\length (\ell) = \length (W)$.
Once we have a standard pair, its image under the map $\mathcal F_0$
is a bunch of unstable curves and some measures living on them.

A nice property of standard pairs is that this image is in fact a 
weighted sum of standard pairs. That is why we call 
weighted sums of standard pairs standard families. Formally, a
standard family is a set $\mathcal G=\{ (W_{a}, \nu_a) \}, a \in \mathfrak A$
of standard pairs and a probability measure $\lambda_{\mathcal G}$
on the index set $ \mathfrak A$. This family defines a probability
measure on $\mathcal M_0$ by
\[ \mu_{\mathcal G} (B) = \int _{\mathfrak{A}} \nu_a (B \cap W_{a}) d \lambda_{\mathcal G}
(a). \]
We will also write $\mathbb E_{\mathcal G}$ for the integral with respect to
$\mu_{\mathcal G}$ and $\mathbb P_{\mathcal G} (A) = \mu_{\mathcal G} (A)$.
Every $x \in W_a$ (for some $a \in \mathfrak A$), chops $W_a$ into two pieces.
The length of the shorter one is denoted by $r_{\mathcal G}(x)$. The
$\mathcal Z$-function of $\mathcal G$ is defined by
\[ \mathcal Z_{\mathcal G} = \sup_{\varepsilon >0} 
\frac{\mu_{\mathcal G}(r_{\mathcal G} < \varepsilon)}{\varepsilon}. \]
Note that if $\mathcal G$ consists of one standard pair, 
then $\mathcal Z_{\mathcal G}=2/|W|$. In any case, we assume
$\mathcal Z_{\mathcal G} < \infty$.

While the unstable curves are expanded due to hyperbolicity, 
they are also cut by the singularities of
$\mathcal F_0.$
An important nice property of the billiard map is that the expansion prevails over the fragmentation.
Namely, the following Growth lemma holds true:
\begin{lemma}
\label{LmInv}
(see \cite[Prop 4.9 and 4.10]{ChD09B}) 
Let $\ell=(W, \rho)$ be some standard pair. Then
\begin{equation}
\label{equ:Markov}
 \mathbb E_\ell (A \circ \mathcal F_0^n)
= \sum_{a} c_{a,n} \mathbb E_{\ell_{an}} (A),
\end{equation}
where $c_{a,n} >0$, $\sum_{a} c_{a,n} =1$; $\ell_{an}
= (W_{an}, \rho_{an})$ are standard pairs such that
$\cup_a W_{an} = \mathcal  F_0^n W$ and $\rho_{an}$ is the
push-forward of $ \rho$ by $\mathcal F_0^n$ up to a multiplicative
factor. Finally, there are universal constants $\varkappa,
C_1$ (depending
only on $\mathcal D$), such that
if $n > \varkappa |\log \length (W)|$, then
\begin{equation}
\label{EqGrowth}
 \sum_{\length (\ell_{an}) < \varepsilon} c_{a,n}
< C_1 \varepsilon. 
\end{equation}
\end{lemma}
We call the decomposition \eqref{equ:Markov} Markov decomposition.
The proof of Lemma~\ref{LmInv} depends on the fact that there are
universal constants $\theta<1, C_2, C_3$ (depending
only on $\mathcal D$) such that for a standard family
$\mathcal G=\{ (W_{a}, \nu_a) \}, a \in \mathfrak A$, and
$\mathcal G_n= \mathcal F_0^n (\mathcal G)$, one has
\[ \mathcal Z_{\mathcal G_n} < C_2 \theta^n \mathcal Z_{\mathcal G}
+ C_3.\]
If we fix some large constant $C_p$ and call a standard family proper
if its $\mathcal Z$ function is smaller than $C_p$, then briefly
one can say that the image of $\mathcal G$ becomes proper in
$\log \mathcal Z_{\mathcal G}$ steps.

The essence of the standard pair technique is that the measures carried
on two proper standard families can be coupled together exponentially fast.
When one of the two standard families is chosen to be $\mu_0$ itself 
(it can be proven that there exists $\mathcal G$ such that
$\mu_{ \mathcal G} = \mu_0$) one obtains the following
Equidistribution statement. Recall that a function $f$ on $\mathcal{M}_0$ is called dynamically 
Holder continuous if there are constants $K>0$ and $\theta<1$ such that 
$|f(x)-f(y)|\leq K \theta^{s(x,y)}$ where $s(x,y)$ is the first number $n$ such that 
either $\mathcal{F}_0^n x$ and $\mathcal{F}_0^n y$ belong to a different scatterer or
$\mathcal{F}_0^{-n} x$ and $\mathcal{F}_0^{-n} y$ belong to a different scatterer.
 
\begin{lemma}[\cite{Ch06} Theorem 4]
\label{lemma:equidistr}
Let $\mathcal G$ be a proper standard family. For any dynamically H\"older
continuous $f$ there exists some $\theta_f <1$ such that for any $n \geq 0$,
\[ \left| \int_{\mathcal M_0} f \circ \mathcal F_0^n  d \mu_{\mathcal G}-
\int_{\mathcal M_0} f d \mu_0 \right| \leq B_f \theta_f^n. \]
\end{lemma}
We will also use standard pairs and standard families on $\mathcal M$
instead of $\mathcal M_0$. If $\ell$ is a standard pair supported on the $m$th
translational copy of the unit torus, then we write $[\ell] =m$.

\subsection{Statistical properties of the Lorentz process}
\label{SSStat}

In \cite{Ch06}, Lemma \ref{lemma:equidistr} is used to prove the invariance 
principle for Lorentz processes of finite horizon. In particular, Lemma 5.4
in \cite{Ch06} implies the following strengthening of \cite{BSC91}
\begin{lemma}
\label{lemma:invprinc}
There is a positive constant $\sigma =\sigma(\mathcal D)$ such that if 
$\mathcal G$ be a proper standard family and $x$ is distributed according to 
 $\mathcal G$,
 then, as $n \rightarrow \infty$,
 $ \left( \frac{X_{\lfloor nt\rfloor} (x)}{\sqrt n} \right)_{t \in [0,1]}$
converges weakly to a Brownian motion with variance $\sigma^2$.
\end{lemma}
It is simple to derive the following continuous time version of Lemma \ref{lemma:invprinc}
(see for example Theorem 5 in \cite{DSzV09}).
\begin{lemma}
\label{lemma:invprinccont}
 Let $\mathcal G$ be a proper standard family, $x$ be distributed according to 
 $\mathcal G$, and write $\hat \sigma= \hat \sigma(\mathcal D) = {\sigma}/\sqrt{ \bar \kappa}$.
 Then, as $n \rightarrow \infty$,
 $ \left( \frac{\hat X (tT) (x)}{\sqrt T} \right)_{t \in [0,1]}$
 converges weakly to a Brownian motion with variance $\hat \sigma^2$.
\end{lemma}

We will use the following result on moderate deviations (called Proposition 3.7 (d) in \cite{DSzV08}).
\begin{lemma}
 \label{lemma3.7d}
Fix some $\delta >0$. There are constants $c_1, c_2$ such that for any dynamically H\"older continuous function $A$,
for any positive integer $n$, for any $R$ with $1 < R < n^{1/6 - \delta}$ and for any standard pair $\ell$ with $|\log length (\ell)| < n^{1/2 -\delta}$,
$$ \mathbb P_{\ell} \left( 
\bigg| \sum_{j=0}^{n-1} A \circ \mathcal F_0^j(x) - n \int A d\mu_0    \bigg| 
> R \sqrt n\right) < c_1 e^{c_2 R^2}.$$
\end{lemma}
Finally, we need a technical estimate (Lemma 11.1 (c) in \cite{DSzV08}).
\begin{lemma}
 \label{lemma11.1c}
 There exists a constant $C$ such that for any standard pair $\ell$ and for any positive integers $n$ and $K$,
$$ \mathbb P_{\ell} \left( \tau_n < \tau^* \text{ and } \tau_n > K n^2 \right)
< \frac{C|\log length (\ell)| }{K^{100} n}. $$
\end{lemma}

\subsection{Local limit theorem for Lorentz processes}

Here we present a variant of the local version of Lemma \ref{lemma:invprinc} (called local
limit theorem for Lorentz processes).

For brevity, let us write 
$\varphi_{\rho} (x) = \frac1{\sqrt{2 \pi} \rho} \exp(-\frac{x^2}{2\rho^2})$ for $\rho >0$,
and $\varphi = \varphi_1$. Further, if $\Sigma$ is a positive definite matrix 
(of size $2 \times 2$ in our case), then
$\varphi_{\Sigma} (x) = \frac1{2 \pi \sqrt{\det(\Sigma)}} \exp(-\frac{x^T \Sigma^{-1} x}{2})$
for $x \in \mathbb R^2$.

Fix some $x,y$ real numbers and some standard pair $\ell$ supported on the zeroth cell. With 
the notation introduced in (\ref{eq:xn}) and (\ref{eq:freeflight}), let us write $\vartheta_n$ for the distribution of 
\[ \left( 
\lfloor X_n(q,v) - x\sqrt n \rfloor , F_n(q,v) - n \bar \kappa - y \sqrt n,
\mathcal F_0^n (q,v)
\right),
\]
where $(q,v)$ is chosen with respect to $\ell$. That is, $\vartheta_n$ is a measure on $\mathbb Z \times
\mathbb R \times \mathcal M_0$.

We also fix the set $\mathcal A \subset \mathbb Z \times
\mathbb R \times \mathcal M_0$ such that $(n,-t,(q,v)) \in \mathcal A$
if and only if $t\geq 0$, the configuration component of $\Phi^t(q+n,v)$ is in the 
zeroth cell, and $|\kappa(q,v)| >t$. 
That is, $\mathcal A$ contains the possible positions of the particle at the last collision time 
before time $0$, when it arrives at the zeroth cell (and also the time spent after the last collision).
Due to the finite horizon assumption $\mathcal A$ is bounded.

\begin{lemma}
 \label{lemma:llt}
There exist some positive definite $2 \times 2$ matrix $\Sigma$ with $\Sigma_{11} = \sigma^2$,
and some finite constants $C, C_1, C_2$
such that for any standard pair $\ell$ with $|\log length (\ell)| < n^{1/4}$ the following hold uniformly.
\begin{itemize}
 \item[(a)] for any real numbers $x,y$,
\[ n \vartheta_n (\mathcal A) \rightarrow 
\bar \kappa \varphi_{\Sigma}(x,y),\]
as $n \rightarrow \infty$ uniformly for $x,y$ chosen from a compact set.
\item[(b)]
for any real numbers $x,y$ and any positive integer $n$,
\[ n \vartheta_n (\mathcal A) < C_1 \varphi_{\Sigma'}(x,y) + C_2 n^{-1/2},\]
where $\Sigma' = C \Sigma. $
\end{itemize}
\end{lemma}

Note that in Lemma \ref{lemma:llt} (a), we fix $x$ and $y$ and then let $n \rightarrow \infty$,
while the estimate in Lemma \ref{lemma:llt} (b) is valid for every $x,y,n$. In particular, we will
use Lemma \ref{lemma:llt} (b) with $x$ or $y$ being roughly of order $n^{0.1}$. In this case clearly
$C_2 n^{-1/2} \gg C_1 \varphi_{\Sigma'}(x,y)$.

Lemma \ref{lemma:llt} (a) is related to the result of \cite{SzV04}
and to Proposition 3.7 (e) in \cite{DSzV08}. 
The main difference is that here, we use an observable that involves the free flight time
and we also take standard pairs as initial measures. The latter means that we compute 
probabilities involving the future conditioned on some event of small probability in the past.
Lemma \ref{lemma:llt} (b) is
related to the last formula on page 834 in \cite{P09}. 
The main difference is again the fact that we use standard pairs as initial measures.
In Appendix \ref{AppLLT} we review the results of \cite{SzV04} for the reader's convenience
and give a proof of Lemma \ref{lemma:llt}.

\subsection{Local time.}
Here we present limit theorems involving the local time at the origin.
\begin{lemma}
 \label{lemma:invprinlocal}
 Let $\mathcal G$ be a proper standard family supported on the zeroth cell and write
 $L_k$ 
for the discrete time spent in the zeroth
 cell up to time $k$.
  If $x$ is distributed according to $\mathcal G$, then
\[ \left( \frac{ X_{\lfloor tn\rfloor}(x)}{\sqrt n}, 
\frac{L_{\lfloor tn\rfloor}(x)}{\sqrt{n}}\right)_{0<t<1}\]
jointly converges to the Brownian motion with variance $\sigma^2$ 
and its local time process at the origin.

\end{lemma}

Lemma \ref{lemma:invprinlocal} is proven for the invariant 
measure in Proposition 3 of \cite{NSz12}. Its proof uses only the local 
limit theorem, which can be extended to proper standard families by 
Proposition 3.7 (e) in \cite{DSzV08} (or by our Lemma~\ref{lemma:llt}). Hence the lemma 
holds in the generality stated above.

The above result obtains local time as the asymptotic number of collisions which occur in the zeroth cell.
We can also count the continuous time. Namely,
let $\hat L_k$ be the continuous
time spent at the zeroth cell between the $k$th and the $(k+1)$st collisions.

\begin{lemma}
 \label{lemma:invprinlocal2}
 Let $\mathcal G$ be a proper standard family supported on the zeroth cell.
  If $x$ is distributed according to $\mathcal G$, then
\[ \left( \frac{ X_{\lfloor tn\rfloor}(x)}{\sqrt n}, 
\frac{ \sum_{k=0}^{\lfloor tn\rfloor -1} \hat L_k(x)}{\sqrt{n}}\right)_{0<t<1}\]
jointly converges to the Brownian motion with variance $\sigma^2$ 
and $\bar \kappa$ times its local time process at the origin.
\end{lemma}

This lemma can be proven by the same argument used in \cite{NSz12} to prove Lemma \ref{lemma:invprinlocal}. 
Namely the proof proceeds by computing the moments of the local time using
the representation $L_n = \sum_{i=0}^{n-1} {\bm 1}_{X_i \in [0,1]}$ and
the local limit theorem (which is finite-dimensional distribution version of 
Theorem \ref{thm:SzV} from Appendix \ref{SS-SV}). The local limit theorem also says that conditioning on
$X_{ns_1}, ..., X_{n s_k}$ being in zeroth cell the asymptotic distribution of 
$(\mathcal F_0^{ns_1}(x), \dots, \mathcal F_0^{ns_k}(x))$ is $\mu_0^k$ (here, $n \rightarrow \infty$ while $s_1, .. s_k \in [0,1]$
are fixed numbers). Thus
$$\mathbb E_{\cG} \left( \prod_{i=1}^k \hat L_{ns_i} \right) 
= \mathbb P_{\cG} (X_{ns_1}, ..., X_{n s_k} \in [0,1]) \bar \kappa^{k}.$$
Here, we have used the fact that $(Counting \times Leb \times \mu_0)(\mathcal{A}) = \bar \kappa$, 
(see \eqref{KatzFla} in Appendix \eqref{sec:ap1}).
With the above observations,
Lemma \ref{lemma:invprinlocal2} can be proven in the same way as Lemma \ref{lemma:invprinlocal}.

\subsection{Brownian meander}
\label{subsect:meander}

Informally, the Brownian meander is a Brownian 
motion on $[0,1]$ conditioned to stay strictly positive
on $(0,1]$.

A formal definition is the following. Consider the Wiener measure 
on $C[0,1]$ conditioned on functions whose minimum is bigger than
$-\varepsilon$. The weak limit of these measures as $\varepsilon \rightarrow 0$
exists and defines the process called Brownian meander (see \cite{DIM77} for more
details).

Let $\mathfrak X_{\rho}$ be a Brownian meander with variance $\rho^2$, 
and $\mathfrak M_{\rho}$ is its maximum
(i.e. $\mathfrak M_{\rho}(t) = \max_{0<s<t} \mathfrak X_{\rho}(s)$) with respect
to some abstract probability measure $P$. For simplicity,
we omit the subscript when $\rho =1$.
The joint distribution function of a Brownian meander and its maximum is
the following:
\begin{equation}
\label{Bmeander}
P(\mathfrak X(1) <x, \mathfrak M(1) <y) = \sum_{k=-\infty}^{\infty}
\left[ \exp \left( -(2ky)^2/2 \right) -
\exp \left( -(2ky+x)^2/2 \right)
\right],
\end{equation}
for any $y \geq x \geq 0$ (see \cite{Ch76}).
In order to prove Theorem \ref{thm:2}, we will need the density in the first coordinate, i.e.
the following quantity:
\[ \phi_{\rho}(x,y) = \lim_{dt \rightarrow 0} \frac{1}{dt} P(\mathfrak X_{\rho}(1) \in [x,x+dt], \mathfrak M_{\rho}(1)<y).\]
An elementary computation yields that for any $y \geq x \geq 0$ ,
\begin{equation}
 \label{eq:meander_density}
 \phi_{\rho} (x,y) =  \sum_{k=-\infty}^{\infty} \frac{2ky+x}{\rho^2} \exp 
 \left( -\frac{(2ky+x)^2}{2 \rho^2} \right).
\end{equation}

\section{Results}
\label{ScRes}
\subsection{Density profile.}
In this section, we formulate our results precisely.
First, we clarify how we emit the particles. 
Let us fix some proper standard family $\mathcal G$ on the phase space
$\mathcal M_0$ to be
the distribution of the particle at its first collision. Then at each time instant
$T_j$ of a Poisson point process on the time interval $[-T,0]$ with intensity $1$, 
we put a Lorentz particle with a position distributed as $\mu_{\mathcal G}$. These initial positions
are independent (and the particles do not interact with each other).
Obviously, not all the standard families are interesting, since for some, $X_1 <0$ almost surely. Thus
for the rest of this paper, we assume that
\begin{equation}
 \label{equ:assumption}
\bar c (\mathcal G) = \lim_{n \rightarrow \infty} L \mathbb P_{\mathcal{G}} (\tau_L < \tau^*)
\end{equation}
exists and is positive (all the admissible standard families satisfy this condition,
see the remark after the proof of Lemma 11.2 in \cite{DSzV08}). We also write
\begin{equation}
\label{EqLimDensity}
c(\mathcal G) = \frac{2 \bar c (\mathcal G) \bar \kappa}{\sigma^2}.
\end{equation}
In the case of the semi-infinite tube, the expected number of particles in the interval $[L, L+1]$
at time 0 is
\[ g_{L,T} = \int_{0}^{T} \mathbb P_{\mathcal{G}} ({\hat X}(t) \in [L, L+1], {\hat \tau}^* >t) dt \]
\begin{theorem}
\label{thm:1}
 $\lim_{L \rightarrow \infty} \lim_{T \rightarrow \infty} g_{L,T} = c(\mathcal G)$,
 where $c(\cG)$ is given by \eqref{EqLimDensity}.
\end{theorem}
In the case of finite tube, we ask a similar
question, namely the density of the particle profile. More precisely,
we are interested in the following quantity
\[ h_{x,L,T} = \int_{0}^{T} \mathbb P_{\mathcal{G}} 
(\hat X (t) \in I^{xL}, \min\{ \hat \tau^*, \hat \tau_L\} >t) dt,\]
where $I^{x,L} = [\lfloor xL \rfloor, \lfloor xL \rfloor +1 ]$, $L>>1$ is the length of the tube and $0<x<1$.

We have the following
\begin{theorem}
\label{thm:2}
For every $0<x<1$,
\[ 
\lim_{L \rightarrow \infty}
\lim_{T \rightarrow \infty}
h_{x,L,T} = c (\mathcal G) (1-x),\]
where $c(\cG)$ is given by \eqref{EqLimDensity}.
\end{theorem}

Finally, we describe the evolution of a density profile when starting from a smooth initial configuration.
Namely, we take a Lorentz tube of length $L$ with absorbing boundaries
and inject particles with rate $\lambda_0$ and with initial measure $\mu_{\mathcal G_0}$ from the left end 
and  with rate $\lambda_1$ and with initial measure $\mu_{\mathcal G_1}$ from the right end. 
We assume that $\mathcal G_i$ are 
proper standard families supported on $\mathcal M_0$ and $\mathcal M_L$ respectively.
Write 
$f_i = \lambda_i c(\mathcal G_i)$
(where $c(\cG_i)$ is given by \eqref{EqLimDensity}). 
Take a non-negative function $f \in C^2[0,1]$
with $f(i) = f_i$ for $i=0,1$. At time zero,
place an independent $Poi(f(k/L))$ amount of particles into 
$\mathcal D |_{[k, k+1]} \times \mathcal S^1$ to some positions chosen by
Lebesgue measure 
for every positive integer $k$ with $k<L$ and also start to emit particles 
from both ends as prescribed above. Let
\[ u_L(t,x) = \mathbb E (\text{number of particles at time }
tL^2 \text{ in cell } \lfloor xL \rfloor),\]
where $\mathbb E$ is the measure generated by the initial particles and
the sources.

\begin{theorem}
\label{thm:4}
The function $ u(t,x) = \lim_{L \rightarrow \infty} u_L(t,x)$ is the solution
of the heat equation with Dirichlet boundary condition, i.e.
\[
u'_t(t,x) = \frac{\hat \sigma^2}2 u''_{xx} (t,x), \quad
u(0,x) = f(x),\quad
u(t,0) = f_0, \quad u(t,1)=f_1. \]
\end{theorem}

We remark that in the case of Theorems \ref{thm:1}, \ref{thm:2}, and \ref{thm:4} the limiting
distribution of the particles in the cell $L$ (and $xL$) is Poissonian with the above computed
parameter. Let us consider for example the setting of Theorem \ref{thm:1}. Note that for any finite
$T$ and $L$, the distribution of particles which have not been absorbed by time $0$ is Poissonian.
Indeed the emitted particles $(T_i, x_i)$ form a Poisson process of 
$\mathbb{M}=[-T, 0]\times \cD\times \cS^1.$ 
Consider a function 
$\mathbb{G}:\mathbb{M}\to (\cD\times \cS^1)\cup\{\infty\}$ where
$\mathbb{G}(t,x)=X(0)$ 
if the particle has not been absorbed by time $0$ and  $\mathbb{G}(t,x)=\infty$ otherwise.
Combining the Mapping and Restriction Theorems for Poisson processes (see \cite{Ki93}, Sections 2.2 and 2.3)
we see that 
$\{\mathbb{G}(T_i, x_i)\}_{G(T_i, x_i)\neq\infty}$ form a Poisson process on $\cD\times \cS^1.$
Since the expected number of particles in $[L, L+1]$ converges as $L\to\infty, T\to\infty$ the limit 
process is also Poisson. Thus Theorems \ref{thm:1}, \ref{thm:2}, and \ref{thm:4} provide the complete
description of limiting distribution. For example the weak Law of Large Numbers for Poisson processes with 
large intensity gives the following.

\begin{corollary}
\label{CrHeat}
Fix $0<\beta<1.$ In the setting of Theorem \ref{thm:4} let $N(t, x, L)$ denote the number of particles with
$|X(t)-tL|<L^\beta$ at time $tL^2.$ Then $\frac{N(x, t, L)}{L^\beta}\to u(t,x)$ in probability as $L\to \infty.$
\end{corollary}

\subsection{Convergence to Brownian meander.}
In order to prove the above results, we need convergence to the Brownian meander, which precisely means the
following.

\begin{theorem}
\label{thm:meander}
The process $\left( \frac{\hat X(tT)}{\sqrt T} \right)_{0<t<1}$
with respect to the measure $\mathbb P_{\mathcal{G}}(.|\hat \tau^* >T)$
converges weakly to the Brownian meander with variance $\hat \sigma^2$.
\end{theorem}

Note that the proof of Theorem 8 in \cite{DSzV08} implies that there exists
some constant $c_1(\mathcal G)>0$ with
 \begin{equation}
 \label{equ:tail}
  \mathbb P_{\mathcal{G}} (\tau^* >N) \sim c_1(\mathcal G) /\sqrt N.
 \end{equation}
Let
\begin{equation}
\label{eq:c1hat}
 \hat c_1(\mathcal G)
= c_1(\mathcal G) \sqrt{\bar \kappa}.
\end{equation}
The following corollaries will be derived from Theorem \ref{thm:meander} in Sections \ref{SSConst1} 
and~\ref{SSConst2} respectively.

\begin{corollary}
\label{LmConst}
Recalling \eqref{equ:assumption} we have
\[ c_1(\mathcal G) = \bar c (\mathcal G) \frac{\sqrt 2}{\sqrt \pi \sigma}.\]
\end{corollary}

\begin{corollary}
\label{eq:c_hat} 
\begin{equation*}
\mathbb P_{\mathcal{G}} (\hat \tau^* > t) \sim
\hat c_1(\mathcal G) / \sqrt t.
\end{equation*} 
\end{corollary}

\subsection{Local Limit Theorems.}
\label{SSLLT}
In order to prove Theorems \ref{thm:2}, and \ref{thm:4}  we will need several new local limit theorems
for the Lorentz particle in the infinite tube.
For this, recall the notation $\phi_{s}(x,y)$, $\mathfrak X$ and $\mathfrak M$
 from Section \ref{subsect:meander}.

\begin{proposition}
 \label{prop:meander_local}
 Fix some $x < y$ positive real numbers. Then
 \[ {T} \mathbb P_{\mathcal G} \left( \lfloor \hat X(T) \rfloor = 
 \lfloor x \sqrt T \rfloor, 
 \forall t, 0<t<T, \hat X(t) \in [0, y \sqrt T] \right) \rightarrow 
\hat c_1(\mathcal G) \phi_{\hat \sigma}(x,y), \]
 as $T \rightarrow \infty$. Furthermore, for any $\delta$, 
 the convergence is uniform
 for $x,y$ such that $\delta < x <x+\delta <y <1/\delta$.
\end{proposition}

\begin{proposition}
 \label{prop:heat_local}
Fix real numbers $x,y$ in $(0,1)$ and $t\in\mathbb{R}_+.$ Let $\cG$ be a proper standard family such that on $\cG$
$\lfloor \hat X(0) \rfloor = \lfloor x  L \rfloor.$ Then
 \[ L^2 \mathbb P_{\mathcal G} \left( \lfloor \hat X(tL^2) \rfloor = 
 \lfloor y L \rfloor, 
 \forall s, 0<s<t, \hat X(sL^2) \in [0, L] \right) \rightarrow 
 \psi(t, x,y) \]
as $L \rightarrow \infty$ where $\psi(t, x,y)$
is the density at $y$ of a Brownian motion at time $t$ which is started from $x$ and killed at $0$ and $1.$
Furthermore, for any $\delta$, 
the convergence is uniform for $ x,y \in [\delta, 1 - \delta]$ and $\delta\leq t \leq 1/\delta.$  
\end{proposition}

\begin{proposition}
 \label{prop:continuous_local}
 Fix some real number $x$. Then
 \[ {\sqrt T} \mathbb P_{\mathcal G} \left( \lfloor \hat X(T) \rfloor = 
 \lfloor x \sqrt T \rfloor \right) \rightarrow 
 \varphi_{\hat \sigma}(x), \]
 as $T \rightarrow \infty$. Furthermore, the convergence is uniform for $x$ chosen from some compact set.
\end{proposition}

\section{Proof of Theorem \ref{thm:1}}
\label{ScHalfLine}

\subsection{Proof for discrete time}
\label{SShalflinepart1}
Here, we prove Theorem \ref{thm:1} without using Brownian meanders (but using Lemma \ref{lemma:invprinlocal}).
In Remark \ref{Remark1} we will sketch another argument using Brownian meanders but not using Lemma \ref{lemma:invprinlocal}.
For brevity, we will write $I=[L,L+1]$.
By Fubini's theorem,
$$ g_{L,T} = \mathbb{E}_{\mathcal{G}}  \int_{0}^{(-T) \wedge \hat \tau^*}
{\bm 1} (X(t) \in I) dt.$$ 
Thus by monotone convergence,
\[ \lim_{T \rightarrow -\infty} g_{L,T} = 
\mathbb{E}_{\mathcal{G}}  \int_{0}^{\hat \tau^*}
{\bm 1} (X(t) \in I) dt. \]
In order to prove that this is convergent as $L\rightarrow \infty$,
let us switch to discrete time first, and prove that the following limit exists
\begin{equation}
 \label{eq:discrete}
 \lim_{L \rightarrow \infty}
\mathbb P_{\mathcal{G}}(\tau_L < \tau^*) 
 \mathbb E_{\mathcal{G}}(\# \{ k < \tau^*: X_k \in I \} | \tau_L < \tau^*)
 = c'(\mathcal G).
\end{equation}
Observe that due to our basic assumption (\ref{equ:assumption}), in order
to prove (\ref{eq:discrete}), it suffices to verify
\begin{equation}
 \label{eq:discrete2}
\mathbb E_{\mathcal{G}}(\# \{ k < \tau^*: X_k \in I \} | \tau_L < \tau^*)
= c_{{ B}}L(1+o(1)).
 \end{equation}
To establish (\ref{eq:discrete2}), write
\begin{eqnarray*} 
&&\mathbb E_{\mathcal{G}}(\# \{ k < \tau^*: X_k \in I \} | \tau_L < \tau^*)\\
&=&
\sum_{m=1}^{\infty}
\mathbb E_{\mathcal{G}}(
\# \{ k < \tau^*: X_k \in I \} {\bm 1} (m=\tau_L) | \tau_L < \tau^* )\\
&=& 
\sum_{m=1}^{\infty} \sum_{n=1}^{\infty} \sum_{\alpha \in J_{m,n}}
c_{\alpha} \mathbb E_{\ell_{\alpha}} (\# \{ k < \tau^* : X_k \in I\}),
\end{eqnarray*}
 where $\ell_{\alpha} = (\gamma_{\alpha}, \rho_{\alpha})$ 
 is a standard pair in the $L$th copy of the
 unit torus ($[\ell_{\alpha}]=L $) and $\length(\ell_{\alpha}) \in [2^{-n}, 2^{-n-1})$, if
 $\alpha \in J_{m,n}$.
We have by definition
\begin{equation}
\label{eq:discrete2.5}
 \sum_{m=1}^{\infty} \sum_{n=1}^{\infty} \sum_{\alpha \in J_{m,n}}
c_{\alpha} =1.
\end{equation}
The growth lemma implies that
\begin{equation}
\label{eq:discrete4}
 \sum_{n>N} \sum_{\alpha \in J_{m,n}} c_{\alpha} <C2^{-N} L 
\end{equation}
holds uniformly in $m$.
Indeed, the term $2^{-N}$ comes from by the growth lemma and 
since we condition on $\{ \tau_L < \tau^* \}$ (which has probability of order $1/L$),
we have a factor of $L$ on the right hand side.
Similarly, Lemma \ref{lemma11.1c} implies
\begin{equation}
\label{eq:discrete5}
\sum_{m>KL^2} \sum_n \sum_{\alpha \in J_{m,n}} c_{\alpha} < C K^{-100} L. 
\end{equation}

We will need the following lemma.
\begin{lemma}
\label{lemma:localtime}
There is a constant $c_B$ and a sequence $\eta_L$ with $ \eta_L/L \to 0$ such that for any standard
pair $\ell$ with $[\ell]=L$ and $|\log \length (\ell)| < \sqrt L$,
\[ |\mathbb E_\ell (\# \{ k < \tau^* : X_k \in I \}) - c_B L| < \eta_L + C |\log \length (\ell)|.\]
For any standard pair $\ell$ with $[l]=L$ and $|\log \length (\ell)| > \sqrt L$,
\[ \mathbb E_\ell (\# \{ k < \tau^* : X_k \in I \})  <  C (L + |\log \length (\ell)|).\]
\end{lemma}
First, we prove that (\ref{eq:discrete}) follows from Lemma \ref{lemma:localtime}.\\
Observe that Lemma \ref{lemma:localtime} implies
\[| \mathbb E_{\ell_{\alpha}} (\# \{ k < \tau^* : X_k \in I\}) - c_B L |
< Cn + o(L) = o(L) \]
uniformly for $\alpha \in J_{m,n}$ with $n<\sqrt L$.
Similarly, for $\alpha \in J_{m,n}$ with arbitrary $m$ and $n$,
\[| \mathbb E_{\ell_{\alpha}} (\# \{ k < \tau^* : X_k \in I\}) - c_B L |
< C(L + n). \]

Using (\ref{eq:discrete2.5}) we conclude that 
in order to prove (\ref{eq:discrete2}), it suffices to show
\begin{equation}
\label{eq:discrete3}
{ 
 \sum_m \sum_{n> \sqrt L} (n+L) \sum_{\alpha \in J_{m,n}} c_{\alpha}
}
= o(L). 
\end{equation}
(\ref{eq:discrete3}) follows by an elementary computation.
Namely, 
(\ref{eq:discrete4}) implies 
\[ \sum_{m<1.99^{\sqrt L}}\sum_{n>\sqrt L} 
(n {+ L}) 
\sum_{\alpha \in J_{m,n}} c_{\alpha} =
o(1)\] and 
\[ \sum_{m>1.99^{\sqrt L}} \sum_{n>\log m/\log 1.99} 
(n {+ L}) 
\sum_{\alpha \in J_{m,n}} c_{\alpha}
=o(1). \] 
On the other hand (\ref{eq:discrete5}) implies
\[\sum_{n>\sqrt L} 
(n {+ L})
\sum_{m>1.99^n} \sum_{\alpha \in J_{m,n}} c_{\alpha} =o(1).\]
Thus, assuming Lemma \ref{lemma:localtime}
 we have proved (\ref{eq:discrete3}) and finished the proof of (\ref{eq:discrete}).\\
\begin{proof}[Proof of Lemma \ref{lemma:localtime}]
Write $\ell=(\gamma, \rho)$ and assume first that $\length (\ell) > \delta$ with some fixed $\delta$. 
Note that Lemma \ref{lemma:invprinlocal} implies that
\[ \frac{\# \{ k < \tau^* : X_k \in I \} }{L}\]
converges weakly to a limit distribution $\xi$, when the initial measure is $\ell$. 
Here, $\xi$ is the local time of a Brownian motion of variance $\sigma^2$ at $1$ up to
its first hitting of the origin assuming that it starts from $1$.
However, we need to prove that the expectations also converge. To this end, choose $K >>1$
and observe that 
\[ \mathbb E_\ell 
\left( \frac{\# \{ k < \tau^* \wedge KL^2 : X_k \in I \}}{L}\right) 
\rightarrow  E (\xi_K),
\]
as $L \rightarrow \infty$, 
where $\xi_K$ is defined in a similar way as $\xi$ except for 'the first hitting of the origin'
being replaced by 'the minimum of the first hitting of the origin and $K$'.
We also have 
\begin{equation}
\label{DefCB} 
\lim_K E (\xi_K) = E (\xi) = c_B.
\end{equation}
It remains to prove that 
\begin{equation}
 \label{lemma11:goal}
\limsup_L \frac1L \mathbb{E}_\ell \left(1_{\tau^* > KL^2  } \# \{ k: KL^2 < k < \tau^*, X_k \in I \} \right)
\end{equation}
is small if $K$ is big.\\
In order to do that, we need one more lemma.\\
Fix a standard pair $\ell'$ in the zeroth cell with $\lim_{L \rightarrow \infty} L\mathbb P_{\ell'} (\tau_{-L} < \tau^*) >0$.
Then there is a rectangle $\mathfrak R$ fully crossed by $\ell'$ and a constant $c$ such that for
any standard pair $\ell''$ fully crossing $\mathfrak R$ and any $L$, we have
$L\mathbb P_{\ell''} (\tau_{-L} < \tau^*) >c$ (see the Appendix of \cite{Ch06}). \\
Now for any $\ell'' =(\gamma '', \rho '')$ and any
$x$ in $\gamma ''$, write $\nu_k$ for the $k$th return to $I$ and $\bar n$ for the
first such time when the curve in $\mathcal F^{\nu_{\bar n}} \gamma ''$ containing 
$\mathcal F^{\nu_{\bar n}} x$ fully crosses $\mathfrak R + L$ (i.e. the translated copy
of $\mathfrak R$ to the $L$th cell). Finally, let us write
$\bar n_0 =0$,
$\bar n_1(x) = \bar n$ and $\bar n_k (x) = \bar n_1 (\mathcal F^{\nu_{\bar n_{k-1}}} x)$.

\begin{lemma}
\label{sublemma}
 There are constants $C,C'$ and $\theta <1$ such that for any standard pair $\ell''$,
\[ \mathbb P_{\ell''} ( \bar n - C |\log \length(\ell'')| >n ) < C' \theta^n.\]
\end{lemma}
Lemma \ref{sublemma} is almost the same as Lemma 11 in \cite{DSzV09}. The only
difference is that in \cite{DSzV09} the curve containing $\mathcal F^{\nu_{\bar n}} x$
can be anywhere in $I$ as long as it has length at least $\delta_0$. The iterated version
of that Lemma (via the coupling algorithm of \cite{Ch06}, as it was also pointed out in 
\cite{DSzV09}) proves our Lemma \ref{sublemma}.\\
Now we apply Lemma \ref{sublemma} to those standard pairs in the 
standard family $\mathcal F^{KL^2} \ell$,
which have not visited the zeroth cell yet. Let $\ell''$ be such standard pair. Then
we have
\begin{eqnarray}
&& \mathbb E_{\ell''} (\# \{ k < \tau^* : X_k \in I \}) - c | \log \length(\ell'')| \nonumber \\
 && \leq \mathbb E_{\ell''} \left( \sum_{j=1}^{\infty} (\bar n_j - \bar n_{j-1} ) {\bm 1}_{ \{ \tau^* > \nu_{\bar n_{j-1}} \} } \right)
=   \sum_{j=1}^{\infty} \mathbb E_{\ell''} \left( (\bar n_j - \bar n_{j-1} ) {\bm 1}_{ \{ \tau^* > \nu_{\bar n_{j-1}} \} } \right) \label{eqthm1discrend}
\end{eqnarray}
Now for any $j$ we can consider Markov decomposition at time $\nu_{\bar n_{j-1}}$. Every standard pair in this decomposition
is longer than a uniform $\delta$ by the definition of $\bar n$. Thus we can apply Lemma \ref{sublemma} and can also neglect the term
$C |\log \length(\ell'')|$. It is not hard to show that if the function $\bar n$ satisfies $\mathbb P (\bar n > n) < C \theta^n$, then 
there is a universal constant $C$ such that $\int_A \bar n d \mathbb P < C [ \mathbb P (A) ]^{0.9}$ for every set $A$.
Thus (\ref{eqthm1discrend}) is bounded by
$$
 C \sum_{j=1}^{\infty} \left( \mathbb P_{\ell''}  ( \tau^* > \nu_{\bar n_{j-1}} ) \right)^{0.9}
\leq C \sum_{j=1}^{\infty} \left( 1- \frac{c}{L} \right)^{0.9j} < CL.
$$
Next, 
$$ \mathbb{E}_\ell \left(1_{\tau^* > KL^2  } \# \{ k: KL^2 < k < \tau^*, X_k \in I \} \right)\leq
L\mathbb{P} (\tau^*>KL^2)
+\mathbb{E}_\ell\left(\ln r_{F^{K L^2} \ell}(x)\right). $$
The first term is $o_{K\to\infty}(L)$ since $\mathbb{P} (\tau^*>KL^2)\to 0$ as $K\to\infty.$ 
On the other hand the second term is $O(1)$ due to the Growth Lemma.
This proves Lemma \ref{lemma:localtime} 
if $\length(\ell)>\delta.$

In the general case let $\hn(x)$ be the first time when $F_0^{\hn}(x)$ belongs to a component which is longer than 
$\delta.$ We then split all visits to the zeroth cell into visits before and after $\hn.$ The later are estimated the same 
way as above. The former contribute at most
$ \mathbb{E}_\ell(\hn)\leq C |\log(\length(\ell))| $
proving Lemma \ref{lemma:localtime} in the general case.
\end{proof}

Finally, we identify the constant in the limit. 
Let us denote a standard two dimensional Brownian motion by $W(t)$. Also write
$\mathcal L_{\varrho}^{a}(T)$ for the local time at position $a$ up to the first hitting of the origin of a one dimensional Brownian motion
with variance $\varrho^2$ starting from $a$. Thus with the notation in \eqref{DefCB}, we have
$$c_B = \mathbb E(\xi) = \mathbb E \left( \mathcal L_{\sigma}^{1}(T) \right) 
 = \frac{1}{\sigma} \mathbb E \left( \mathcal L_{1}^{1/\sigma}(T) \right). $$
Observe that due to the Ray-Knight theorem (see \cite{R63} and \cite{Kn63}), we have
$$c_B = \frac{1}{\sigma} \mathbb E \| W ({1}/{\sigma}) \|^2 = \frac{2}{ \sigma^2}. $$
Thus for the constant defined in (\ref{eq:discrete}), we have
$c'(\mathcal G) = 2 \bar c (\mathcal G)/ \sigma^2$.

\subsection{Proof for continuous time}
\label{SShalflinepart2}

Our proof for the case of continuous time is similar to the proof in Subsection \ref{SShalflinepart1}.
Thus we only highlight the differences.

Recall the notation $\hat L_k$ introduced in Section \ref{SSStat}.
Note that in order to finish the proof of Theorem \ref{thm:1}, it suffices to verify the following
analogue of (\ref{eq:discrete2})
\begin{equation}
 \label{eq:halfline2_1}
\mathbb E_{\mathcal{G}}
\left( \sum_{k < \tau^*} \hat L_k | \tau_L < \tau^* \right)
= \bar \kappa \frac{2}{\sigma^2} L(1+o(1)).
 \end{equation}
Indeed, (\ref{eq:halfline2_1}) and the computations in Subsection \ref{SShalflinepart1} yield
\begin{eqnarray*}
 &&\lim_{L \rightarrow \infty} \lim_{T \rightarrow \infty} g_{L,T} = 
\lim_{L \rightarrow \infty} \mathbb{E}_{\mathcal{G}}  \int_{0}^{\hat \tau^*}
{\bm 1} (X(t) \in I) dt = \\
&& \lim_{L \rightarrow \infty} \mathbb P_{\cG} (\tau_L < \tau^*) 
\mathbb{E}_{\mathcal{G}}  \left( \sum_{k < \tau^*} \hat L_k | \tau_L < \tau^* \right)
= \lim_{L \rightarrow \infty} \frac{\bar c (\cG)}{L} \bar \kappa \frac{2}{\sigma^2} L(1 + o(1))
= c(\cG).
\end{eqnarray*}
The proof of (\ref{eq:halfline2_1}) is similar to that of (\ref{eq:discrete2}) 
except that Lemma \ref{lemma:localtime} should be replaced by the following
\begin{lemma}
\label{lemma:localtime2}
There is a sequence $\eta_L$ with $ \eta_L/L \searrow 0$ such that for any standard
pair $\ell$ with $[\ell]=L$ and $|\log \length (\ell)| < \sqrt L$,
\[ \Bigg| \mathbb E_\ell \left( \sum_{k < \tau^*} \hat L_k | \tau_L < \tau^* \right) - \bar \kappa \frac{2}{\sigma^2} L \Bigg|
 < \eta_L + C |\log \length (\ell)|.\]
For any standard pair $\ell$ with $[\ell]=L$ and $|\log \length (\ell)| > \sqrt L$,
\[ \mathbb E_\ell \left( \sum_{k < \tau^*} \hat L_k | \tau_L < \tau^* \right)   <  C (L + |\log \length (\ell)|).\]
\end{lemma}

The proof of Lemma \ref{lemma:localtime2} follows the same lines as the proof of Lemma \ref{lemma:localtime}
except that instead of referring to Lemma \ref{lemma:invprinlocal}
we use Lemma \ref{lemma:invprinlocal2}. This completes the proof of Theorem \ref{thm:1}.

\section{Brownian meander as a limit}
\label{ScMeander}

\subsection{Proof of Theorem \ref{thm:meander}}
First, we prove the theorem for discrete time, i.e. the statement that
$\left( \frac{X_{\lfloor tN \rfloor }}{\sqrt N} \right)_{0<t<1}$
with respect to the measure $\mathbb P_{\mathcal{G}}(.|\tau^* >N)$
converges weakly to a Brownian meander.\\
Let us begin with a lemma. Let 
$\tau_{-L}$ denote the first time the particle reaches $-L$ for the system in the doubly infinite tube without the absorption at
the origin.

\begin{lemma}
\label{lemma:1}
There exist some constants $\theta<1$ and $C<\infty$, such that
for $K\leq n^{10}$ and for a proper standard family $\mathcal G$, with $n$
large enough,
\[ \mathbb{P}_{\mathcal G} \left(
\min \{ \tau_n, \tau_{-n} \} > K n^2 \right) \leq \theta^K + 
\frac{CK}{n^{1000}}.\]
For $K\geq n^{10} $ and $K$ large enough,
\[ \mathbb{P}_{\mathcal G} \left(
\min \{ \tau_n, \tau_{-n} \} > K n^2 \right) \leq \theta^{K^{0.8}} +
\frac{C}{K^{99}}.\]
\end{lemma}

\begin{proof}
To prove the first statement it suffices to show that if $\ell$ is a standard pair with $\length(\ell)>n^{-1000}$
then
\[ \mathbb{P}_{\ell} \left(
\min \{ \tau_n, \tau_{-n} \} > K n^2 \right) \leq \theta^K + 
\frac{CK}{n^{1000}}.\]
We prove this by induction on $K$. For $K=1$, the statement
 is true due to the invariance principle for Lorentz process (Lemma \ref{lemma:invprinc}).  
 Here $\theta$ is the probability that the maximum of a Brownian motion
 up to time $1$ is smaller than $1$. To apply Lemma \ref{lemma:invprinc} 
 we use the fact that by Lemma \ref{LmInv}
 the image of $\Prob_\ell$ becomes proper after $\bar{K}\log N$ iterations while due to finite horizon property
 the particle travels distance $O(\log N)$ during the time $\bar{K} \log N.$
 
 Assume that
 the statement is true for some $K$. Then with the notation
\begin{eqnarray*}
 &&\mathbb{P}_{\mathcal G} \left(
\min \{ \tau_n, \tau_{-n} \} > (K+1) n^2 \right)=\\
&&\mathbb{P}_{\mathcal G} \left(
\min \{ \tau_n, \tau_{-n} \} > K n^2 \right)
 \mathbb{P}_{\mathcal G} \left(
 \min \{ \tau_n, \tau_{-n} \} > (K+1) n^2 |
\min \{ \tau_n, \tau_{-n} \} > K n^2 \right)\\
&=&I*II,
\end{eqnarray*}
$I$ is estimated by the inductive hypothesis. In order to bound $II$
we use the Markov decomposition at time $Kn^2$. For standard pairs which are longer
than $n^{-1000}$, we simply use the statement for $K=1$ while the contribution of the short pairs is 
estimated by Lemma \ref{LmInv}. We obtain
\[ I*II < 
\left( \theta^K + \frac{CK}{n^{1000}} \right) 
\theta + \frac{C'}{n^{1000}} 
< \theta^{K+1} + \frac{C(K+1)}{n^{1000}}, \]
assuming that $C$ is large enough.

To prove the second statement we use the first one with 
$n_{new}=K^{0.1}$ and $K_{new}=K^{0.8}$. Thus
\begin{eqnarray*}
&& \mathbb{P}_{\mathcal G} \left(
\min \{ \tau_n, \tau_{-n} \} > K \right)
< \mathbb{P}_{\mathcal G} \left(
\min \{ \tau_{K^{0.1}}, \tau_{K^{0.1}} \} > K \right)\\
&&= \mathbb{P}_{\mathcal G} \left(
\min \{ \tau_{n_{new}}, \tau_{-n_{new}} \} > K_{new} n_{new}^2 \right)
 < \theta^{K^{0.8}} + \frac{C}{K^{99}}. \qedhere
\end{eqnarray*}
\end{proof}

\begin{lemma}
\label{lemma:meander_eps}
 For any $\varepsilon>0$, with $N$ large enough, we have
 \[\mathbb P_{\mathcal{G}} \left( \tau_{\varepsilon \sqrt N} > \varepsilon N |
 \tau^* > N \right) < C \theta^{1/\varepsilon} \]
\end{lemma}

\begin{proof}
We have
 \begin{eqnarray*}
  &&\mathbb P_{\mathcal{G}} \left( \tau_{\varepsilon \sqrt N} > \varepsilon N ,
 \tau^* > N \right) <\\
 &&\mathbb P_{\mathcal{G}} \left( \tau^* > \varepsilon N/2 \right)
 \mathbb P_{\mathcal{G}} \left( \min\{ \tau_{\varepsilon \sqrt N}, \tau^* \}
 > \varepsilon N | \tau^* > \varepsilon N/2 \right) = I*II.
 \end{eqnarray*}
$I$ is bounded by $c/\sqrt{\varepsilon N}$ by (\ref{equ:tail}). In order to estimate $II$, we use
Markov decomposition at time $\varepsilon N/2$ and the first part of Lemma
\ref{lemma:1} to conclude
\begin{eqnarray}
&& II < \frac{1}{N^{100}} + \sum_{\alpha} c_{\alpha} \mathbb P_{\ell_{\alpha}}
\left( \min\{ \tau_{\varepsilon \sqrt N}, \tau_{-\varepsilon \sqrt N} \}
 > \frac{\varepsilon N}{2} \right) \nonumber \\
 &<& \frac{1}{N^{100}} + \theta'^{\frac{1}{2\varepsilon}} + 
 \frac{c}{\varepsilon^{1001} N^{500}} < \theta^{1/\varepsilon},
 \nonumber
\end{eqnarray}
 where the $\ell_{\alpha}$'s are those standard pairs in the $\varepsilon N/2$-fold
 iterate of $\mathcal G$, which are longer than $N^{-100}$ (or more precisely,
 their shifted version to the zeroth cell). The statement follows.
\end{proof}
We are now ready to prove the discrete time version of Theorem \ref{thm:meander}. Namely, let us fix some distance in
the space of probability measures on $C([0,1]).$ Take a small $\delta.$ Choose $\eps$ so that $C \theta^{1/\varepsilon}<\delta$
and such that the Brownian Motion started from $\eps$ and conditioned on not hitting $0$ before time 1 is $\delta$-close in distribution
to the Brownian meander. Then by Lemma \ref{lemma:meander_eps} there is a set $\Prob(\cdot|\tau^*>N)$
measure at least $1-\delta$ where $\tau_{\eps\sqrt{N}}<\eps N.$ If $x$ is in this set and $t>\eps N$ then we can write
$$ \frac{X_{\lfloor tN\rfloor}}{\sqrt{N}}=\frac{X_{\lfloor tN\rfloor-\tau_{\eps\sqrt{N}}}(x_{\tau_{\eps\sqrt{N}}})}{\sqrt{N}}$$
and observe that by the invariance principle for the Lorentz process the distribution of the RH'S is close to the distribution
of the Brownian Motion started from $\eps.$ Applying the conditioning we obtain that the distribution of 
$\frac{X_{\lfloor tN \rfloor }}{\sqrt N} $ under $\Prob(\cdot|\tau^*>N)$ is close to the distribution of the Brownian meander.

The extension of the convergence to continuous time is straightforward.
The finite horizon condition implies that the time needed for the first
$\varepsilon N$ collisions is bounded by $\kappa_{\max} \varepsilon N$.
In the discrete time interval $[\varepsilon N, N]$ we used the invariance
principle for Lorentz process (Lemma \ref{lemma:invprinc});
now we can apply its continuous time 
counterpart (Lemma \ref{lemma:invprinccont}). Thus we have finished the proof
of Theorem \ref{thm:meander}.

\subsection{Proof of Corollary \ref{LmConst}}
\label{SSConst1}
Let us write
$$ A_{N} = \{ \tau^* >N\} \text{ and } B_{N, \varepsilon} = \{ \tau_{\varepsilon \sqrt N} < \tau^* \}. $$
Using Lemma \ref{lemma:meander_eps} we conclude that
$$ \lim_{N \rightarrow \infty} \mathbb P_{\mathcal G} (A_N | B_{N, \varepsilon}) $$
is asymptotic (as $\varepsilon \rightarrow 0$) 
to the probability that the minimum of a Brownian motion of variance $\sigma^2$
up to time $1$ is bigger than $ - \varepsilon$. Thus an elementary computation shows
\begin{eqnarray}
&&\lim_{\varepsilon \rightarrow 0} \lim_{N \rightarrow \infty} \sqrt N \mathbb P_{\mathcal G} (A_N \cap B_{N, \varepsilon}) 
=\lim_{\varepsilon \rightarrow 0} \lim_{N \rightarrow \infty} \sqrt N \mathbb P_{\mathcal G} (B_{N, \varepsilon})  \mathbb P_{\mathcal G} (A_N | B_{N, \varepsilon}) \nonumber \\
&=&\lim_{\varepsilon \rightarrow 0} \frac{\bar c (\mathcal G) }{\varepsilon} \frac{\sqrt 2}{\sqrt \pi \sigma} \varepsilon  
= \bar c (\mathcal G) \frac{\sqrt 2}{\sqrt \pi \sigma}. \label{eq:constident}
\end{eqnarray}
On the other hand, the definition of $c_1(\mathcal G)$ and Theorem \ref{thm:meander} imply that
\begin{eqnarray}
\nonumber
\lim_{\varepsilon \rightarrow 0} \lim_{N \rightarrow \infty} \frac{\mathbb P_{\mathcal G} (A_N \cap B_{N, \varepsilon})}{c_1(\mathcal G)/ \sqrt N}
&=& \lim_{\varepsilon \rightarrow 0} \lim_{N \rightarrow \infty} \frac{\mathbb P_{\mathcal G} (A_N \cap B_{N, \varepsilon})}{\mathbb P_{\mathcal G} (A_N)}\\
&=& \lim_{\varepsilon \rightarrow 0} \lim_{N \rightarrow \infty} \mathbb P_{\mathcal G} ( B_{N, \varepsilon} | A_N) = 1. 
\nonumber 
\end{eqnarray}
The statement follows.

\subsection{Proof of Corollary \ref{eq:c_hat}}
\label{SSConst2}
Analogously to the proof of Corollary \ref{LmConst}, let 
$$ A = \{ \tau^* >T\}, \quad B = \{ \hat \tau_{\frac{\varepsilon \sqrt T}{\kappa_{\max}}} < \hat \tau^* \} \text{ and } C = \{  \hat \tau_{\frac{\varepsilon \sqrt T}{\kappa_{\max}}} < \varepsilon T\}. $$
By the definition of $\kappa_{\min}$ and $\kappa_{\max}$ and by Lemma \ref{lemma:meander_eps}, we have
\begin{eqnarray*}
 &&\mathbb P_{\cG} (\bar C|A)  =
\frac{ \mathbb P_{\cG} \left( \hat \tau_{\frac{\varepsilon \sqrt T}{\kappa_{\max}}}> \varepsilon T \text{ and } \hat \tau^* >T  \right)}{\mathbb P_{\cG} \left( \hat \tau^* >T \right)}
\leq
\frac{ \mathbb P_{\cG} \left( \tau_{\frac{\varepsilon \sqrt T}{\kappa_{\max}}}> \frac{\varepsilon T}{\kappa_{\max}} \text{ and } 
 \tau^* > \frac{T}{\kappa_{\max}}  \right)}{\mathbb P_{\cG} \left( \tau^* > \frac{T}{\kappa_{\min}} \right)}\\
&&=
\mathbb P_{\cG} \left( \tau_{\frac{\varepsilon \sqrt T}{\kappa_{\max}}}> \frac{\varepsilon T}{\kappa_{\max}} \left|
 \tau^* > \frac{T}{\kappa_{\max}}  \right.\right)
\sqrt{\frac{\kappa_{\max}}{\kappa_{\min}}}(1+o_T(1))
\leq C\theta^{1/\varepsilon}.
\end{eqnarray*}
Since $\mathbb P_{\cG} (ABC) = \mathbb P_{\cG} (AC)$, we conclude 
\begin{equation}
\label{eq:corollary6_1}
 \lim_{\varepsilon \rightarrow 0} \lim_{T \rightarrow \infty} \frac{\mathbb P_{\cG} (ABC)}{\mathbb P_{\cG} (A)} = 1.
\end{equation}

Now we can use Markov decomposition at $\tau_{\frac{\varepsilon \sqrt T}{\kappa_{\max}}}$ and Lemma \ref{lemma:invprinccont}
to deduce the following analogue of (\ref{eq:constident}):
\begin{equation}
\label{eq:corollary6_2}
\lim_{\varepsilon \rightarrow 0} \frac{\kappa_{\max}}{\varepsilon} \lim_{T \rightarrow \infty} {\mathbb P_{\cG} (A |BC)} =   \frac{\sqrt 2}{\sqrt \pi \frac{\sigma}{\sqrt{\bar \kappa}}}.
\end{equation}
Notice that by Lemma \ref{lemma11.1c} we have
\begin{equation}
\label{eq:corollary6_3}
\lim_{\varepsilon \rightarrow 0} \lim_{T \rightarrow \infty} {\mathbb P_{\cG} (C|B)} = 1.
\end{equation}
Since by definition
$$ \lim_{\varepsilon \rightarrow 0} \lim_{T \rightarrow \infty} \frac{\varepsilon \sqrt T}{\kappa_{\max}} {\mathbb P_{\cG} (B)} = \bar c(\cG), $$
we can finish the proof by combining (\ref{eq:corollary6_3}), (\ref{eq:corollary6_2}), (\ref{eq:corollary6_1}) and Corollary \ref{LmConst}.

\section{Proofs of the Local Limit Theorems.}
\label{ScLLT}
Here we prove Proposition \ref{prop:meander_local}. 
The proofs of Proposition \ref{prop:heat_local} and Proposition \ref{prop:continuous_local}
are similar but easier so we leave them to the reader.

\subsection{Upper bound}
\label{subsect:upperbd}

First, we prove the upper bound.
The strategy of our proof is the following. First, 
we write 
\begin{equation}
\label{eq:collnumber}
N= \frac{T}{\bar \kappa}, \quad
N_1 = (1- \delta_t)\frac{T}{\bar \kappa}.
\end{equation}
We choose $\delta_t, \delta_s$ small positive numbers, and chop the interval $[0, y \sqrt{ N \bar \kappa}]$
into pieces of length $\delta_s y \sqrt{ N \bar \kappa}$. Using Theorem \ref{thm:meander}, we can estimate the probability
of arriving into one of these intervals at time $N_1 = (1-\delta_t)N$. For the upper bound, we simply 
omit the condition that the particle should stay in the interval $[0, y \sqrt N]$ between time
$(1-\delta_t)N$ and $T$. 
Fix a large constant $A$. We expect that typically there are $n$ collisions with 
\begin{equation}
n \in \mathcal I = [ \delta_tT/\bar \kappa - A \sqrt T, \delta_tT/\bar \kappa + A \sqrt T] \cap \mathbb N
\end{equation}
between discrete time $N_1$ and continuous time $T.$  
The contribution
of $n$'s chosen from $\mathcal I$ 
can be computed with Lemma \ref{lemma:llt} (a). The contribution of
$n$'s from $\mathbb N \setminus \mathcal I$ is small, which can be verified by using
Lemma \ref{lemma:llt} (b).

We use the following simple property of Brownian meanders proven in Appendix \ref{AppMeander}.
\begin{lemma}
\label{lemma:meanderprop}
 The Brownian meander satisfies the following. 
\begin{equation}
 \phi_{\hat \sigma}(x,y) = \lim_{\delta_t \rightarrow 0} \lim_{\delta_s \rightarrow 0} \sum_{h = 1}^{\lfloor 1/ \delta_s \rfloor} \wp_{1,h} \wp_{2,h},
\end{equation}
where 
\[ \wp_{1,h} = P \bigg( \mathfrak X_{\hat \sigma}(1- \delta_t) \in [hy\delta_s, (h+1)y\delta_s], \mathfrak M_{\hat \sigma}(1- \delta_t) \leq y \bigg),\]
\[  \wp_{2,h} = \varphi_{\hat \sigma \sqrt{ \delta_t}}(x- y_h)\]
and $y_h \in [hy\delta_s, (h+1)y\delta_s]$ is arbitrary.
\end{lemma}

Let us fix some small $\varepsilon>0,$ choose small positive numbers $\delta_t, \delta_s$  
(to be specified later) and write
$\{ \ell_{h,\alpha}\}_{\alpha \in \mathfrak A_{h}}$ for the set of standard pairs in 
$\mathcal F^{N_1}_* \mathcal G$ 
satisfying 
\[\tau^* (x) > N_1 \text{ and } [\ell_{h,\alpha}] \in [h \delta_s y \sqrt{ N \bar \kappa}, (h+1) \delta_s y \sqrt{ N \bar \kappa}]. \]
By Theorem \ref{thm:meander}, we have the Markov decomposition
\[ \mathbb P_{\mathcal G} (x\in \mathcal F^{-N_1} \mathcal B \text{ and } \tau^* (x) > N_1) =
\sum_{h=1}^{1/\delta_s}
\sum_{\alpha \in \mathfrak A_{h}} c_{h,\alpha} \ell_{h,\alpha} (\mathcal B),
\]
where $\mathcal B \subset \mathcal M$ is measurable and
$ \sum_{\alpha \in \mathfrak A_{h}} c_{h,\alpha}$
is asymptotic to
\begin{equation}
\label{eq:verticalchop}
 \frac{\sqrt{\bar \kappa} c_1(\cG)}{\sqrt{T}} P \left( \mathfrak{X}_{\sigma}(1-\delta_t) \in
 \Big[ h \delta_s y \sqrt{\bar \kappa}, {(h+1) \delta_s y \sqrt{\bar \kappa}} \Big],
 \mathfrak{M}_{\sigma}(1-\delta_t) < y \sqrt{\bar \kappa} \right)
\end{equation}
for every $h$.
Since $\mathfrak X_{\sigma} / \sqrt{\bar \kappa}$ has the same distribution as $\mathfrak X_{\hat \sigma}$, we conclude that
with the notation of Lemma \ref{lemma:meanderprop},
\begin{equation}
\label{eq:upperbdfirstterm}
 \sum_{\alpha \in \mathfrak A_{h}} c_{h,\alpha} \sim \frac{1}{\sqrt T} \sqrt{\bar \kappa} c_1(\mathcal G) \wp_{1,h}. 
\end{equation}
Now let us fix some standard pair $\ell_{h, \alpha}
= (\gamma_{h, \alpha}, \rho_{h, \alpha})$. We want to compute the probability of arriving in
$[ \lfloor x \sqrt T \rfloor,\lfloor x \sqrt T \rfloor +1] $ at continuous time 
$T$ assuming that at discrete time $N_1$ the point is distributed according to $\ell_{h, \alpha}$. Clearly, we will need 
to control
the continuous time spent during discrete time $N_1$. Thus let us write
\[ f_{h, \alpha} = \sum_{i=0}^{N_1-1} | \kappa 
\left( \mathcal F^{-i} (q,v)\right) 
|\]
with some fixed $(q,v) \in \gamma_{h, \alpha}$. Even though $f_{h, \alpha}$ depends on 
the choice of $(q,v)$, in order to keep notation simple, we pretend it does not and explain
at the end of the proof how the argument should be modified to treat non-constant $f_{h, \alpha}$.
Observe that by Lemma \ref{lemma3.7d}, the complement of the event
\begin{equation}
\label{eq:festimate}
|f_{h, \alpha} - \bar \kappa N_1| =  |f_{h, \alpha} - (1- \delta_t) T |
< N^{0.6}
\end{equation}
has superpolynomially small $\mathbb P_{\cG}$-probability. Thus we can assume that (\ref{eq:festimate})
is true.

By the growth lemma,  we can also neglect the contribution of standard pairs $\ell_{h, \alpha}$ with
\begin{equation}
\label{RelLong}
|\log \length (\ell_{h, \alpha })| > N^{1/4}.
\end{equation}
Thus we can assume that Lemma \ref{lemma:llt} is applicable to $\ell_{h, \alpha}$.
Since $f_{h, \alpha }$ is not exactly equal to $(1- \delta_t) T$, we need to adjust the definition of 
$\mathcal I$. Namely, let us write
\begin{equation}
\label{eq:Iha}
\mathcal I_{T,h,\alpha} = 
[ (T - f_{h, \alpha })/\bar \kappa - A \sqrt T, 
(T - f_{h, \alpha })/\bar \kappa + A \sqrt T] \cap \mathbb N
\end{equation}
Now by Lemma \ref{lemma:llt} (a), for every $n \in \mathcal I_{T, h,\alpha}$ with the
notation $n = \lfloor (T - f_{h, \alpha })/\bar \kappa \rfloor +m$,
we have
\begin{eqnarray}
&&q_{T, h, \alpha, n} := \mathbb P_{\ell_{h, \alpha }} \bigg(
(
X_n - x \sqrt T + [\ell_{h, \alpha }],
F_n-T+f_{h, \alpha },
\mathcal F_0^n (q,v)
) \in \mathcal A
\bigg) 
\nonumber
\\
&& \sim \frac{\bar \kappa}{n} \varphi_{\Sigma}
\left( \frac{ x \sqrt T - [\ell_{h, \alpha }]}{\sqrt n}, \frac{m \bar \kappa}{\sqrt n} \right). \label{eq:2dlclt}
\end{eqnarray}
Note that by (\ref{eq:festimate}), 
$$ \min_h \min_{\alpha \in \mathfrak A_h} \mathcal I_{T, h, \alpha}$$ tends to infinity at a linear speed with $T$.
Thus Lemma \ref{lemma:llt} a also implies that
the convergence in (\ref{eq:2dlclt}) is {\bf uniform} in $h, \alpha$ 
satisfying \eqref{eq:festimate} and \eqref{RelLong} and $n \in \mathcal I_{T, h, \alpha}$.
Also, we have 
\begin{equation}
\label{eq:sqrtnasympt}
  \sqrt n \sim \sqrt{\frac{\delta_t T}{\bar \kappa}}
\end{equation}
uniformly for $h, \alpha$ and $n \in \mathcal I_{T, h, \alpha}$.
Hence with the notation
\[ y_{h,\alpha} = \frac{[\ell_{h, \alpha}]}{\sqrt T} \in [ h \delta_s y,  (h+1) \delta_s y],\]
we also have
\begin{equation}
\label{eq:intfirstcoord}
 \frac{ x \sqrt T - [\ell_{h, \alpha }]}{\sqrt n} \sim \sqrt{ \bar \kappa} \frac{x - y_{h,\alpha} }{\sqrt{\delta_t}}
\end{equation}
uniformly for $h, \alpha$.
Thus summing up the estimation in (\ref{eq:2dlclt}) for $n \in \mathcal I_{T, h, \alpha}$, substituting 
a Riemann sum with the integral and using (\ref{eq:intfirstcoord}), we obtain that
$$ 
\sum_{n \in \mathcal I_{T, h, \alpha}} q_{T, h, \alpha, n}
\sim \frac{\bar \kappa^2}{\delta_t \sqrt{ T} } \int_{-A}^A \varphi_{\Sigma}
\left( 
\sqrt{ \bar \kappa} \frac{x - y_{h,\alpha}}{\sqrt{\delta_t}},
 \frac{\bar \kappa^{3/2}}{\sqrt{\delta_t}} y
\right) dy
$$
uniformly for $h, \alpha$. 
With the notation of Lemma \ref{lemma:meanderprop}, by choosing $y_h = y_{h,\alpha}$, we have
\[ 
\wp_{2,h} = 
\frac{ \sqrt{\bar \kappa}}{\sigma \sqrt{2 \pi \delta_t}}
\exp \left( 
- \frac{\bar \kappa (x-y_{h,\alpha})^2}{2 \sigma^2 \delta_t}
\right).
\]
Thus for any fixed positive numbers $\varepsilon, \delta_t, \delta_s$, by choosing a large 
$A = A(\varepsilon, \delta_t, \delta_s)$, we conclude
\begin{equation}
\label{eq:deltatupperbd}
\bigg|
\sum_{n \in \mathcal I_{T, h, \alpha}} q_{T, h, \alpha, n}
- \frac{1}{\sqrt{T} } 
\wp_{2,h}
 \bigg|
 < \frac{\delta_s \varepsilon}{\sqrt T}
\end{equation}
for $T$ large enough (uniformly in $h, \alpha$).\\
Now, we want to bound
\begin{equation}
\label{eq:prop1leftside}
 T \mathbb P_{\mathcal G} \left( \lfloor \hat X(T) \rfloor = 
 \lfloor x \sqrt T \rfloor, 
 \forall t, 0<t<T, \hat X(t) \in [0, y \sqrt T] \right)
\end{equation}
from above by
\begin{equation}
\label{eq:upperbdmainterm}
T \sum_{h=1}^{1/\delta_s} 
\sum_{\alpha \in \mathfrak A_{h}} c_{h,\alpha}
 \sum_{n \in \mathcal I_{T, h, \alpha}} q_{T, h, \alpha, n}.
\end{equation}
Performing the summation over $h$, using (\ref{eq:upperbdfirstterm}), (\ref{eq:deltatupperbd}) and Lemma \ref{lemma:meanderprop},
we conclude that (\ref{eq:upperbdmainterm}) is close to 
\[ c_1(\mathcal G) \sqrt{\bar \kappa} \phi_{\hat \sigma} (x,y) = {\hat c}_1(\mathcal G) \phi_{\hat \sigma} (x,y).\]
(Here  $\hat c_1$ is defined by \eqref{eq:c1hat}. See also Corollary \ref{eq:c_hat}.)
More precisely, the closeness means $\varepsilon$-closeness when $\delta_t= \delta_t(\varepsilon)$,
$\delta_s= \delta_s(\delta_t, \varepsilon)$, $A= A(\delta_s, \delta_t, \varepsilon)$, $T_0 = T_0(A, \delta_s, \delta_t, \varepsilon)$
are chosen appropriately and $T>T_0$.

In order to conclude the upper bound, we need to check two technical details which we treat in two separate Lemmas.

\begin{lemma}
\label{lemma:upperbdlargens}
Given $\varepsilon$ there exist constants $A$ and $T_0$ such that if $T\geq T_0$ then
the contribution of $n \notin \mathcal I_{T,h,\alpha}$ is bounded by $\varepsilon / \sqrt T$.
\end{lemma}

\begin{proof}

Clearly, for $n <  n_1 = \delta_t T / (2 \kappa_{\max})$ and for $n > n_2 = 2 \delta_t T / \kappa_{\min}$
we have $q_{T,h,\alpha,n} =0$. 
Applying Lemma \ref{lemma3.7d} to the function $| \kappa|$,
we conclude that the contribution of indices $n \in [n_1, n_2]$ with 
\[ |n - (T - f_{h, \alpha })/\bar \kappa | > T^{0.6}\]
is bounded from above by a superpolynomial term:
\[ \sum_{n: n_1 < n < n_2, |n - (T - f_{h, \alpha })/\bar \kappa | > T^{0.6}} q_{T,h,\alpha,n} 
< 
C T e^{ - cT^{0.2}}.\]
For the remaining $n$'s, we will use Lemma \ref{lemma:llt} (b). Because of symmetry reasons, we only need to compute the contribution of
\[ n \in \mathcal I'_{T, h,\alpha} = [(T - f_{h, \alpha })/\bar \kappa + A \sqrt T , (T - f_{h, \alpha })/\bar \kappa + T^{0.6} ] \cap \mathbb N.\]
Thus, with the notation 
$n = \lfloor (T - f_{h, \alpha })/\bar \kappa \rfloor +m$, we have
\begin{eqnarray}
 && \sum_{n \in \mathcal I'_{T, h,\alpha} } q_{T,h,\alpha,n} \nonumber \\
&&< \sum_{n \in \mathcal I'_{T, h,\alpha} }
\frac{C_1}{n} \varphi_{\Sigma'} \left( \frac{ x \sqrt T - [\ell_{h, \alpha }]}{\sqrt n}, \frac{m \bar \kappa}{\sqrt n} \right)
+ \frac{C_2}{n^{3/2}} \label{eq:moderatedev1}
\end{eqnarray}
Since (\ref{eq:sqrtnasympt}) and (\ref{eq:intfirstcoord}) hold uniformly for $ n \in \mathcal I'_{T, h,\alpha}$, 
we conclude that there are some positive finite constants $c = c(\delta_t),C_i = C_i(\delta_t)$ for $i = 3,4,5$ such that
(\ref{eq:moderatedev1}) is bounded by
\begin{eqnarray*}
 &&\frac{C_3}{T} \left( \sum_{m=  A \sqrt T}^{T^{0.6} } \exp(-c m^2/T) \right) + C_4 T^{-0.9}\\
&& < \frac{C_3}{T} \left( \sum_{m=  A \sqrt T}^{\infty } \left( \exp \left( - \frac{c}{\sqrt T} \right) \right) ^m \right) + C_4 T^{-0.9}\\
&& < \frac{C_3}{T} \exp(- c A) \frac{1}{1-\exp \left( - \frac{c}{\sqrt T} \right) } + C_4 T^{-0.9} <
\frac{C_5 e^{-cA}}{ \sqrt T}
\end{eqnarray*}
for $T$ large enough.
Thus by choosing $A = A(\varepsilon, \delta_t)$ large enough we can guarantee $C_5 e^{-cA} < \varepsilon$.
\end{proof}

\begin{lemma}
 \label{lemma:upperbdrealfha}
The above argument remains valid for $(q,v)$-dependent $f_{h,\alpha}$
\end{lemma}

\begin{proof}

Note that by the H\"older continuity of $|\kappa|$, 
for every $\bar \epsilon >0$ there exists some $\delta >0$
such that 
\[ \dist((q,v), (q',v')) < \delta \text{ implies } |f_{h, \alpha}(q,v) - f_{h, \alpha}(q',v')| < \bar \epsilon. \]
For any given $\delta >0 $ we can chop the standard pairs to smaller pieces by introducing artificial singularities
so that any standard pair is shorter than $\delta$. 

Thus taking the real $f_{h, \alpha}(q,v)$ instead of the constant $ \bar f_{h, \alpha}$
in (\ref{eq:2dlclt}), we have
\begin{eqnarray}
&& \mathbb P_{\ell_{h, \alpha }} \bigg(
(
X_n - x \sqrt T + [\ell_{h, \alpha }],
F_n-T+\bar f_{h, \alpha },
\mathcal F_0^n (q,v)
) \in \mathcal A_{\bar \epsilon}
\bigg) \nonumber \\
&\leq & \mathbb P_{\ell_{h, \alpha }} \bigg(
(
X_n - x \sqrt T + [\ell_{h, \alpha }],
F_n-T+ f_{h, \alpha } (q,v),
\mathcal F_0^n (q,v)
) \in \mathcal A
\bigg) \label{eq:realfha} \\
&\leq & \mathbb P_{\ell_{h, \alpha }} \bigg(
(
X_n - x \sqrt T + [\ell_{h, \alpha }],
F_n-T+\bar f_{h, \alpha },
\mathcal F_0^n (q,v)
) \in \mathcal A^{\bar \varepsilon}
\bigg), \nonumber
\end{eqnarray}
where
$$\begin{array}{lll}
  \mathcal A_{\bar \epsilon} & =  &\{ (x,y,\omega): \forall y', |y-y'| < \bar \epsilon, (x,y',\omega) \in \mathcal A \}   \\
 \mathcal A^{\bar \epsilon} & =  &\{ (x,y,\omega): \exists y', |y-y'| < \bar \epsilon, (x,y',\omega) \in \mathcal A \}. 
\end{array}$$
Thus applying the Local limit theorem for $\mathcal A_{\bar \epsilon}$ and $\mathcal A^{\bar \epsilon}$, and using the 
fact that $\partial \mathcal A$ has zero measure 
we see that the for $T$ large enough, the ratio of (\ref{eq:realfha}) and (\ref{eq:2dlclt})
is in $[1- \varepsilon, 1+\varepsilon]$ (by choosing $\bar \varepsilon = \bar \varepsilon(\varepsilon)$ and
$\delta = \delta(\bar \varepsilon, \varepsilon)$ small enough). With this adjustment, one can apply the above argument 
for $(q,v)$-dependent $f_{h,\alpha}$.
\end{proof}

\subsection{Lower bound}
\label{subsect:lowerbd}

We use the notation of Subsection \ref{subsect:upperbd}.
Note that our previous argument for the upper bound was in fact  an asymptotic equality except for
one point: when we substituted (\ref{eq:prop1leftside}) by (\ref{eq:upperbdmainterm}). Thus the 
lower bound (and hence Proposition \ref{prop:meander_local}) will be established whenever we prove the following
statement.

For every $\varepsilon >0$ there exist $\delta_t= \delta_t(\varepsilon)$,
and $T_0 = T_0(\delta_t)$
such that for every $T > T_0$ and for every $h$ and $\alpha$,
\begin{eqnarray}
&&
\mathbb P_{\ell_{h,\alpha}} 
\left( 
\lfloor \hat X (T- f_{h, \alpha}) \rfloor = \lfloor x \sqrt T \rfloor,
 \exists s < T- f_{h, \alpha}: \hat X (s) \notin [0, y \sqrt T]
\right)  \label{eq:lowerbd} \\
&& < \varepsilon / \sqrt T \nonumber
\end{eqnarray}

In the remaining part of the subsection we prove (\ref{eq:lowerbd}).\\
Let us fix some $\ell_{h, \alpha} = (\gamma_{h, \alpha}, \rho_{h, \alpha})$. In order to keep the notation 
simple, we will discard the subscript and simply write 
$\ell = (\gamma, \rho)= \ell_{h, \alpha} = (\gamma_{h, \alpha}, \rho_{h, \alpha})$, $f =f_{h, \alpha}$.\\
Let us denote by $\tilde n_1$ the smallest integer
(a random variable w.r.t. $\ell$) such that at time $N_1 + \tilde n_1$ the particle
is outside of the tube segment $[0, y \sqrt T]$.
Let us write
\[ \mathbb Q_\ell (.) = \mathbb P_\ell  
\left( . | \sum_{i=0}^{\tilde n_1 -1} |\kappa \circ \mathcal F^i| < T- f
\right),\]
i.e. $\mathbb Q_\ell$ is the conditional probability under the condition that the particle leaves the tube segment
$[0, y \sqrt T]$ before continuous time $T$. We have the Markov decomposition at time $\tilde n_1$
\[ \mathbb Q_\ell (\mathcal F^{\tilde n_1} (q,v) \in \mathcal B) =  
\sum_{\beta \in \mathfrak B} c_{\beta} \ell_{\beta} (\mathcal B).\]
Let us write $\mathcal T_{\beta}$ for the remaining continuous time until time $T$. More precisely, 
observe that for fixed $\beta$, for every $(q,v) \in l$ with $\mathcal F^{\tilde n_1} (q,v)$ being on the standard pair $\ell_{\beta}$,
$\tilde n_1$ is the same. Thus using that common $\tilde n_1$, we can write
\[ \mathcal T_{\beta} = T -  \sum_{i=0}^{N_1 + \tilde n_1 -1} |\kappa \circ \mathcal F^{-i} (q,v)|\]
with some $(q,v) \in \ell_{\beta}$. This definition depends slightly on the choice of $(q,v)$, but for simplicity, 
we will ignore this issue (similarly to $T- f_{h, \alpha}$ in Subsection \ref{subsect:upperbd} - but this case is simpler
since we only need to prove that (\ref{eq:lowerbd}) is small thus we can enlarge $\mathcal A$ instead of proving the analogue of
Lemma \ref{lemma:upperbdrealfha}). 
Clearly the event $\tilde n_1 < (T - f) / \kappa_{\min}$ has full $\mathbb Q_\ell$ probability, 
thus the growth lemma implies
$$
\sum_{\beta: |\log \length (\ell_{\beta})| > T^{1/8}} c_{\beta} < \frac{C T \exp( - c T^{1/8})}
{\mathbb{P}_\ell(\text{the particle leaves }[0, y\sqrt{T}])}.
$$
Since we want to prove that (\ref{eq:lowerbd}) is less than $\varepsilon / \sqrt T$, we can clearly neglect the contribution 
of standard pairs $\ell_{\beta}$ with $|\log \length (\ell_{\beta})| > T^{1/8}$. In particular, we can assume that
all of our standard pairs are long enough in the sense that $|\log \length (\ell_{\beta})| < (\sqrt T)^{1/4}$ thus
Lemma \ref{lemma:llt} and Lemma \ref{lemma3.7d} are applicable with $n \geq \sqrt T$.
Finally note that by definition $[\ell_{\beta}]$ is $\kappa_{\max}$-close to either $\lfloor y \sqrt T \rfloor$ or $-1$.

When estimating the probability 
\begin{equation}
\label{eq:lowerbdbeta}
\mathbb P_{\ell_{\beta}} \bigg(
\lfloor \hat X (\mathcal T_{\beta}) \rfloor = 
\lfloor x \sqrt T \rfloor
\bigg) 
\end{equation}
we distinguish two cases.

{\bf Case 1} $\mathcal T_{\beta} < T^{0.99}$\\
In this case we estimate a probability of a very unlikely event. Whence it is enough to estimate the
'global probability' instead of its local version. Namely, we can use Lemma \ref{lemma3.7d}.
Note that if $\mathcal T_{\beta} < \frac{ \min \{ x, y-x \}}{ \kappa_{\min}} \sqrt T$, then the probability
we are computing is zero. Thus we can assume that the number of collisions before time $\mathcal T_{\beta}$
is bigger than $c \sqrt T$.\\
Let us write $n_0 = T^{0.995}$. Note that it is impossible to have $n > n_0$ collisions during continuous time $\mathcal T_{\beta}$
due to the finite horizon condition.
If there are $n$ collisions with $c \sqrt T < n < n_0$ before time $\mathcal T_{\beta}$, then it is very unlikely that
the particle travels distance $\min \{ x, y-x \} \sqrt T$ in discrete time $n$. Thus we can bound the 
probability in (\ref{eq:lowerbdbeta}) by
$$
\sum_{n = c \sqrt T}^{n_0} \mathbb P_{\ell_{\beta}} ( X_n > C \sqrt T ) = 
\sum_{n = c \sqrt T}^{n_0} \mathbb P_{\ell_{\beta}} \left( X_n > C \sqrt n \sqrt{ \frac{T}{n} } \right),
$$
which is bounded by
$$
\sum_{n = c \sqrt T}^{n_0} C \exp (- c \frac{T}{n} ) < C T^{0.995} \exp (- c T^{0.005} )
$$
due to Lemma \ref{lemma3.7d}.\\

{\bf Case 2} $T^{0.99} < \mathcal T_{\beta} < \delta_t T$ 
where $\delta_t$ is from \eqref{eq:collnumber}.

Similarly to the estimations in Subsection \ref{subsect:upperbd}, we write
\begin{equation*}
\mathcal I_{T,\beta} = 
[ \mathcal T_{\beta}/\bar \kappa - \sqrt{ T }, 
 \mathcal T_{\beta}/\bar \kappa + \sqrt{ T } ] \cap \mathbb N
\end{equation*}
and use Lemma \ref{lemma:llt} (b) to derive that for every $n \in \mathcal I_{T, h,\alpha}$ with the
notation $n = \lfloor (T - f_{h, \alpha })/\bar \kappa \rfloor +m$,
we have
\begin{eqnarray*}
&&q_{T, \beta, n} := \mathbb P_{\ell_{\beta}} \bigg(
(
X_n - x \sqrt T + [\ell_{\beta}],
F_n-\mathcal T_{\beta},
\mathcal F_0^n (q,v)
) \in \mathcal A
\bigg) < \\
&& <  \frac{C_1}{n} \varphi_{\Sigma'}
\left( \frac{ x \sqrt T - [\ell_{\beta }]}{\sqrt n}, \frac{m \bar \kappa}{\sqrt n} \right)
+ \frac{C_2}{n^{3/2}}.
\end{eqnarray*}
Note that we also have
$$
\bigg| \frac{ x \sqrt T - [\ell_{\beta }]}{\sqrt n} \bigg| > \frac{ \sqrt{\bar \kappa} \min \{ x, y-x \}}{2} \sqrt{ \frac{T}{\mathcal T_{\beta}}}.
$$
Thus by simply using $\varphi_{\Sigma'}(x,y)< C \exp(-cx^2)$, we obtain
$$ 
\sum_{n \in \mathcal I_{T,\beta}} q_{T, \beta, n} <
\frac{C \sqrt T}{{\mathcal T_{\beta}}} \exp \left( -c  \frac{T}{\mathcal T_{\beta}} \right)
+ C_2 \sqrt{ T}  \mathcal T_{\beta}^{-3/2}.
$$
Since the function $x \exp(- c x)$ tends to zero as $x \rightarrow \infty$ and $\mathcal T_{\beta} \in [T^{0.99}, \delta_t T]$, 
we have
$$
\sqrt T \sum_{n \in \mathcal I_{T,\beta}} q_{T, \beta, n} <
C\frac{T}{\mathcal T_{\beta}}
\exp \left( -c  \frac{T}{\mathcal T_{\beta}} \right)
+ C_2 T^{1 - 0.99 * 3/2} < \varepsilon
$$
assuming that $\delta_t = \delta_t (\varepsilon)$ is small enough and $T$ is large enough.

For the estimation of the remaining possible collision numbers $n\not\in I_{T,\beta}$ 
we essentially need to repeat the proof of Lemma \ref{lemma:upperbdlargens}. 
Namely, observe that $\mathcal T_{\beta}^{0.6} > \sqrt T$ and 
by using that $\varphi_{\Sigma'}(x,y)< C \exp(-cy^2)$ we can bound the contribution of the $n$'s in 
\[ 
\mathcal I'_{T,\beta} = 
[ \mathcal T_{\beta}/\bar \kappa + \sqrt{ T }, 
 \mathcal T_{\beta}/\bar \kappa + \mathcal T_{\beta}^{0.6} ] \cap \mathbb N
\]
by
\[ \frac{C}{\mathcal T_{\beta}} \left( \sum_{m=  \sqrt T}^{\mathcal T_{\beta}^{0.6} } \exp(-c m^2/ \mathcal T_{\beta}) \right) + 
C \mathcal T_{\beta}^{0.6 - 3/2}\\
<  C \frac{1}{ \sqrt \mathcal T_{\beta}}  \exp \left( -c \frac{T}{\mathcal T_{\beta}} \right)  + C T^{- 0.99 * 9/10} \]
for $T$ large enough. As before, this expression is less than $\varepsilon / \sqrt T$ for $\delta_t =  \delta_t(\varepsilon)$ small 
and $T = T(\delta_t)$ large enough.
Finally, the case 
$$|n - \mathcal T_{\beta}/\bar \kappa| > \mathcal T_{\beta}^{0.6} $$
is treated exactly the same way as in Lemma \ref{lemma:upperbdlargens}. We have finished the proof of (\ref{eq:lowerbd})
and hence that of Proposition \ref{prop:meander_local}.

\section{Proof of Theorem \ref{thm:2}}
\label{ScFiniteTube}
Since under the condition that a particle does not return to the origin
it still diffuses, we expect that the main contribution to 
$h_{x,L} = \lim_{T \rightarrow \infty} h_{x,L,T}$ comes from the time interval $[\delta t^2, t^2/\delta]$.
Thus with the notation $I^{xL} = [\lfloor xL \rfloor, \lfloor xL \rfloor +1 ]$, define
\[ I_{x,L, \delta} = 
\int_{\delta L^2}^{L^2/\delta} \mathbb P_{\mathcal{G}} 
({\hat X}(t) \in I^{xL}, \min\{ {\hat \tau}^*, {\hat \tau}_L\} >t) dt.\]
Using Proposition \ref{prop:meander_local} (with $T,x$ and $y$ being replaced by $t, x L/ \sqrt t$ 
and $L/ \sqrt t$, respectively), we obtain
\begin{equation} 
I_{x, \delta} = \lim_{L \rightarrow \infty} 
I_{x,L, \delta} =  \lim_{L \rightarrow \infty} 
 \hat c_1 (\mathcal G) \int_{\delta L^2}^{L^2/\delta} \frac{1}{t} \phi_{\hat \sigma} 
\left( \frac{x L}{\sqrt t}, \frac{L}{\sqrt t} \right)dt
\end{equation}
Writing $s=\frac{L}{\sqrt t}$ we have
\begin{equation}
\label{eq:meanderintegral}
I_{x,\delta} = 2 \hat c_1 (\mathcal G) \int_{\sqrt \delta}^{1/\sqrt \delta} \frac{1}{s} \phi_{\hat \sigma}
(xs,s) ds.
\end{equation}
Substituting formula (\ref{eq:meander_density}), we conclude
\begin{eqnarray*}
I_{x, \delta} &=&
\frac{2 \hat c_1 (\mathcal G)}{\hat \sigma^2}
\int_{\sqrt{\delta}}^{1/\sqrt{\delta}}
\sum_{k=-\infty}^{\infty}
(2k+x) \exp \left( -\frac{(2k+x)^2s^2}{2 \hat \sigma^2} \right)ds\\
&=&
\frac{2 \hat c_1 (\mathcal G)}{\hat \sigma^2}
\sum_{k=-\infty}^{\infty}
\int_{\sqrt{\delta}}^{1/\sqrt{\delta}}
(2k+x) \exp \left( -\frac{(2k+x)^2s^2}{2 \hat \sigma^2} \right)ds.
\end{eqnarray*}

In order to establish that the equilibrium profile is linear, it remains to prove two lemmas.

\begin{lemma}
\label{LmLimBox}
\[ I_x:= \lim_{\delta \rightarrow 0} I_{x,\delta}
=\frac{\hat c_1 (\mathcal G) \sqrt{2 \pi}}{\hat \sigma} (1-x).\]
\end{lemma}

This Lemma is proved in Appendix \ref{AppMeander}.

\begin{lemma}
\label{lemma:small_and_big_t}
\begin{equation}
\lim_{\delta \rightarrow 0}
\lim_{L \rightarrow \infty}
\int_{[0,\delta L^2] \cup [L^2/ \delta, \infty)}
\mathbb P_{\mathcal{G}} 
(\hat X(t) \in I^{xL}, \min\{ \hat \tau^*, \hat \tau_L\} >t) dt =0 
\end{equation}
\end{lemma}

\begin{proof}
For the case $t \in [0, \delta L^2]$, let us write
\begin{eqnarray}
&& \int_{0}^{\delta L^{2}}
\mathbb P_{\mathcal{G}} 
(\hat X(t) \in I^{xL}, \min\{ \hat \tau^*, \hat \tau_L\} >t) dt \nonumber \\
&\leq&
\int_{0}^{\delta L^{2}}
\mathbb P_{\mathcal{G}}
( \tau_{\frac{xL}{2}} < \tau^*) 
\mathbb P_{\mathcal{G}}
( \hat X(t) \in I^{xL} | \tau_{\frac{xL}{2}} < \tau^*) 
dt \label{eq:verysmallt}
\end{eqnarray}
We have $\mathbb P_{\mathcal{G}}
( \tau_{\frac{xL}{2}} > \tau^*) = O(1/L)$ for fixed $x$ by (\ref{equ:assumption}). 
On the other hand the argument used in Section \ref{subsect:lowerbd} to prove \eqref{eq:lowerbd} shows that
for every given $x \in (0,1)$ and $\varepsilon >0$, there exists a $\delta>0$ such that for large enough $L$ 
and for any $t < \delta L^2$,
$$ \mathbb P_{\mathcal{G}}
( \hat X(t) \in I^{xL} | \tau_{\frac{xL}{2}} < \tau^*) < \frac{\varepsilon}{L}. $$
Substituting these estimations to (\ref{eq:verysmallt}), we obtain
$$\lim_{\delta \rightarrow 0}
\lim_{L \rightarrow \infty}
\int_{0}^{\delta L^2}
\mathbb P_{\mathcal{G}} 
(\hat X(t) \in I^{xL}, \min\{ \hat \tau^*, \hat \tau_L\} >t) dt =0. $$

Next, consider the case of $t > L^2/\delta.$
Similarly to the proof of Lemma \ref{lemma:1}, we get that
for any $t$ with $ L^2/\delta <t <L^{12}$, 
\[  \mathbb{P}_{\mathcal G} \left(
\min \{ \hat \tau^*, \hat \tau_L \} > t/2 \right) \leq \frac{C}{L} 
\left( \theta^{t/L^2} +
\frac{Ct}{L^{1002}} \right).\]
Indeed, the term $1/L$ comes from the fact that $\hat \tau^* > L^2$, while the other
term on the right hand side comes from the same argument as the proof of
the first case of Lemma \ref{lemma:1} with $n=L$ and $K=t/(2L^2) - 1$ 
(possibly with some different $\theta$ and $C$). Let us denote by $\tilde n$ the smallest $k$
when $F_k > t/2$. Applying Markov decomposition at time $\tilde n$ and using Proposition 
\ref{prop:continuous_local} we conclude that there is some $\theta<1$ and $C < \infty$
such that
$$
\mathbb{P}_{\mathcal G} \left(
\min \{ \hat \tau^*, \hat \tau_L \} > t/2, \hat X(t) \in I^{xL} \right) \leq \frac{C}{L} 
\left( \theta^{t/L^2} +
\frac{Ct}{L^{1002}} \right) \frac{1}{\sqrt t} \leq 
\frac{C}{L^2} \left( \theta^{t/L^2} + \frac{Ct}{L^{1002}} \right) 
$$
For $t>L^{12}$
we simply use the second part of Lemma \ref{lemma:1} to conclude
\begin{eqnarray*}
&& \lim_{L\rightarrow \infty}
\int_{L^2/\delta}^{\infty}
\mathbb P_{\mathcal{G}}
(\hat X(t) \in I^{xL}, \min\{ \hat \tau^*, \hat \tau_L\} >t) dt \\
&<& \lim_{L\rightarrow \infty} \left(
\frac{C}{L^2} \int_{L^2/\delta}^{L^{12}} \left(\theta^{t/L^2} + \frac{Ct}{L^{1002}} \right) dt
+ \int_{L^{12}}^{\infty} \left( \theta^{t^{0.8}/L^{1.6}} + \frac{CL^{198}}
{t^{99}} \right) dt
\right) < \theta^{1/\delta}.
\end{eqnarray*}
The proof of Lemma \ref{lemma:small_and_big_t} is complete.
\end{proof}

The last step in the proof is the identification of the constant.
Corollary \ref{LmConst} implies 
\[  \frac{\hat c_1 (\mathcal G) \sqrt{2 \pi}}{\hat \sigma} = 
 \frac{c_1 (\mathcal G) \sqrt{\bar \kappa } \sqrt{2 \pi}}{ \sigma / \sqrt{\bar \kappa}} 
= \frac{2 \bar c (\mathcal G) \bar \kappa}{\sigma^2} = c( \mathcal G).\]
Thus we have finished the proof of Theorem \ref{thm:2}.

\begin{remark}
\label{Remark1}
The argument used in this section can be adapted to prove Theorem \ref{thm:1}.
Observe that the main contribution to 
$h_{x, L, T}$ and $g_{x L, T}$ comes from particles whose age is of order $x^2 L^2.$ 
If $x\ll 1$ then such particles do not have enough time to reach the $L$-th cell so that
$ h_{x, L, T} \approx g_{x L, T}$. One can make this argument rigorous by combining  
(\ref{lemma11:goal}) with the argument of the present section
thus obtaining
another proof of Theorem \ref{thm:1} using Brownian meanders but not using Lemma \ref{lemma:invprinlocal}.
This also explains the fact that the constants appearing in Theorems \ref{thm:1} and \ref{thm:2} are the same.
\end{remark}

\begin{remark}
\label{Remark2}
Note that $c_1(\mathcal E)$ is computed on page 277 of \cite{DSzV08}, where $\mathcal E$ is the 
special standard family for which $\mu_{\mathcal G} = \mu_0$. Using their formula and Corollary \ref{LmConst}
we conclude that
$$ c(\mathcal E) = 2.$$
Note also that in the case of general $\cG$, there is no explicit formula for $c(\cG)$.
\end{remark}

\section{Proof of Theorem \ref{thm:4} }
\label{ScHeat}
In this section, we prove Theorem \ref{thm:4}.
Let us write
\begin{eqnarray*}
 u^{\delta}_L(t,x) &=& \sum_{k=\delta L}^{(1- \delta) L} f(k/L) \mu_k(\lfloor \hat X(tL^2)\rfloor=\lfloor x L \rfloor, \hat X(s)\in [0, L] \text{ for } s\leq tL^2)\\
&+& \int_{\delta L^2}^{t L^2} \lambda_0 
\Prob_{\cG_0} (\lfloor \hat X(s) \rfloor=\lfloor xL \rfloor,   \hat X(u)\in [0, L] \text{ for } u \leq s) ds \\
&+& \int_{\delta L^2}^{t L^2} \lambda_1 
\Prob_{\cG_1} (\lfloor \hat X(s) \rfloor=\lfloor - (1- x)L \rfloor,   \hat X(u)\in [- L, 0] \text{ for } u \leq s) ds
\end{eqnarray*}
where $\mu_k$ denotes the measure $\mu_0$ shifted to the cell $k.$

Note that by definition,
$$ u_L(t,x) = u^{0}_L(t,x)$$
For every fixed small positive $\delta$, we can apply 
Propositions \ref{prop:meander_local} and \ref{prop:heat_local} (as in the derivation of 
(\ref{eq:meanderintegral})) to conclude
\begin{eqnarray}
&& u^{\delta}(t,x):= \lim_{L \rightarrow \infty} u^{\delta}_L(t,x) =
\int_{\delta}^{1-\delta} f(z) \psi(t, z, x) dz+ \nonumber\\
&& 2 \lambda_0 \hat c_1(\mathcal G_0) \int_{1/\sqrt t}^{1/ \sqrt \delta} \frac{1}{s} 
\phi_{\hat \sigma} \left(xs,s \right) ds
+ 2 \lambda_1 \hat c_1(\mathcal G_1) \int_{1/\sqrt t}^{1/ \sqrt \delta} \frac{1}{s} 
\phi_{\hat \sigma} \left((1-x)s,s \right) ds.  \label{eq:heatintegral}
\end{eqnarray}
Applying Proposition \ref{prop:continuous_local}
for the first term and Lemma \ref{lemma:small_and_big_t} 
for the second and third terms,
we conclude that (\ref{eq:heatintegral}) also holds for $\delta =0$ (with the 
identification $1/0 = \infty$). Namely,
\begin{eqnarray*}
&& u(t,x)= \lim_{L \rightarrow \infty} u_L(t,x) =
\int_{0}^{1} f(z) \psi(t, z, x) dz+ \nonumber\\
&& 2 \lambda_0 \hat c_1(\mathcal G_0) \int_{1/\sqrt t}^{\infty} \frac{1}{s} 
\phi_{\hat \sigma} \left(xs,s \right) ds
+ 2 \lambda_1 \hat c_1(\mathcal G_1) \int_{1/\sqrt t}^{\infty} \frac{1}{s} 
\phi_{\hat \sigma} \left((1-x)s,s \right) ds =: I_1+I_2 +I_3.
\end{eqnarray*}
We need to check that all the integrals $I_1,I_2,I_3$ satisfy 
the heat equation. It is a well known fact about Gaussian densities that
$$ \frac{\partial}{\partial t} \psi(t, z, x) = 
\frac{\hat \sigma ^2}{2} \frac{\partial^2}{\partial x^2} \psi(t, z, x)$$
Since $\psi(t, z, 0) = \psi(t, z, 1)=0$, $I_1$ satisfies the heat equation of
Theorem \ref{thm:4} with constant $0$ boundary conditions.
Due to symmetry reasons, it remains to apply the following result proven in Appendix \ref{AppMeander}.

\begin{lemma}
\label{LmBoundaryLayer}
$$ u(t,x)=
2 \lambda_0 \hat c_1(\mathcal G_0) \int_{1/\sqrt t}^{\infty} \frac{1}{s} 
\phi_{\hat \sigma} \left(xs,s \right) ds
$$
solves the following Cauchy problem for the heat equation
$$u'_t(t,x) = \frac{\hat \sigma^2}2 u''_{xx} (t,x), \quad
u(0,x)=0,\quad u(t, 0)=f_0, \quad u(t,1) =0. $$
\end{lemma}

\appendix
\section{Proof of Lemma \ref{lemma:llt}}
\label{AppLLT}

\subsection{Local Limit Theorem of Sz\'{a}sz and Varj\'{u}.} 
\label{SS-SV}
Before proving Lemma~\ref{lemma:llt} we briefly summarize the main statement
of \cite{SzV04}.

Take a bounded H\"older function
$f : \mathcal M_0 \rightarrow \mathbb R^d$ (in our case, $d=2$)
and consider the smallest closed subgroup of $\mathbb R^d$  which supports the
values of the function $f - r$ for some constant $r$.
Denote this subgroup by $S(f)$.
Let us also write $f \sim g$ if
there exists
some measurable $h$ with $f-g = h-h\circ \mathcal F_0$ (that is, $f$ and $g$ are cohomologous).
With the notation
\[ M(f) = \cap_{g: g \sim f} S(g) \]
we say that the function $f$ is minimal if $M(f) = S(f)$.
We say that $f$ is non-degenerate if $span(M(f))= \mathbb R^d$. In this case,
there exists some lattice $\mathcal L$ of dimension $d' \leq d$ such that
$M(f)$ is isomorphic to $ \mathcal L \times \mathbb R^{d - d'}.$

Fix some vector $k \in \mathbb R^d$ and a sequence $k_n \in S(f)+nr$
such that
\begin{equation}
\label{eq:k_nspeed}
 \Big\| \frac{k_n - n \mu_0(f)}{\sqrt n} -k \Big\| \rightarrow 0.
\end{equation}
Choose the initial point $x \in \mathcal M_0$ according to the measure $\mu_0$
and denote by $\upsilon_n$ the distribution of the triple
$$ \left( x, \sum_{i=0}^{n-1} f \circ \mathcal F^i (x) - k_n, 
\mathcal F_0^n (x)
\right).$$
Thus the measure $\upsilon_n$ is supported on $\mathcal M_0 \times S(f) 
\times \mathcal M_0$.
Finally, we denote by $U$ the uniform measure 
(i.e. product of counting and Lebesgue measures)
on $S(f)$. Here $U$ is normalized so that 
constant in this uniform measure is chosen in such a way that
$U(B(R)) \sim$Leb$(B(R))$ for large $R$ 
(in order words, the product of the usual
counting and Lebesgue measures is multiplied by 
$\vol (\mathbb R^{d'} / \mathcal L)$).
\begin{theorem}(\cite{SzV04})
\label{thm:SzV}
Assume that the function $f$ is minimal and non-degenerate.
Then there exists
some positive definite $d \times d$ matrix $\Sigma_f$ such that
$n^{d/2}\upsilon_n$ converges vaguely to
$$\varphi_{\Sigma_f} (k) \mu_0 \times U \times \mu_0.$$
Furthermore, for any fixed sequence $\delta_n \searrow 0$ and compact set
$\mathcal K \subset \mathbb R^d$
the above convergence is uniform in the choice of $k_n$
and $k \in \mathcal K$ if the sequence in (\ref{eq:k_nspeed}) is bounded by
$\delta_n$.
\end{theorem}

In the proof of Lemma \ref{lemma:llt}, we will use certain constructions from 
the papers
\cite{Ch06}, \cite{Ch07}, \cite{P09}
\cite{SzV04} and \cite{Y98} without giving the original details.
Our proof consists of three major steps.

\subsection{Proof of Lemma \ref{lemma:llt} (a) for the invariant measure}
\label{sec:ap1}

First, let us replace the standard pair $\ell$ by the measure $\mu_0$ and prove the
convergence
\begin{equation}
\label{eq:A1first}
 n \vartheta_n(\mathcal A) \rightarrow \varphi_{\Sigma}(x,y) \bc_{\mathcal A}.
\end{equation}
with some constant $\bc_{\mathcal A}$.
We are going to apply Theorem \ref{thm:SzV}. First,
take the function
\[ f(q,v)=(\psi(q,v), |\kappa(q,v)| - \bar \kappa ),\]
 where $\psi$ is the discretized version
of $\Pi \kappa$ and $\Pi$ is the projection to the horizontal direction
(exactly as in Section 5 of \cite{SzV04}). Clearly, the smallest closed subgroup
of $\mathbb R^2$ that supports the values of $f$ is $\mathbb Z \times \mathbb R$.
In order to apply Theorem \ref{thm:SzV}, we need to check that the function $f$ is minimal.
Note that by Theorem
3.1 of \cite{SzV04}, there exists a minimal function
in each cohomology class. In particular, there is some $\tilde f
= (\tilde f_1, \tilde f_2) \sim f$
with $S(\tilde f) = M(f)$.\\
Also note that the billiard flow can be represented as a suspension flow over $(\mathcal M_0, \mathcal F_0)$ with
roof function $| \kappa |$. With this identification, the usual notation for the phase space of
the billiard flow is
\[  \Omega = \{ (x,t): x \in \mathcal M_0, 0 \leq t < |\kappa (x)| \}. \]
It also makes sense to take $(x,t) \in \Omega$, where $t > |\kappa ( x)|$ with the identification
$(x,t) = (\mathcal F_0 x, t - |\kappa (x)|)$. With this notation the billiard flow $\Phi_{|\kappa|}^t$ acts on $\Omega$
by $\Phi_{|\kappa|}^t(x,s) = (x, s+t)$ and preserves the measure $\mu_0 \times Leb$.
We need the following result.
\begin{lemma}
\label{lemma:ap2mixing}
For arbitrary positive constant $b$ the suspension
flow
$\Phi_{|\kappa| + b}^t$
over $(\mathcal M_0, \mathcal F_0)$ with
roof function $|\kappa| + b$ is weak mixing.
\end{lemma}

\begin{proof}
In fact, for billiard flows one knows much stronger result. Namely the flow enjoys stretched exponential
decay of correlations \cite{Ch07}. The proof of this fact given in \cite{Ch07} relies only on the properties
of so called temporal distance function. Namely given $x$ and $y$ such that both
$v_1=W^u_{loc}(x)\cap W^s_{loc}(y)$ and $v_2=W^s_{loc}(x)\cap W^u_{loc}(y)$ exist one can define
$$\Delta_{|\kappa|}(x,y)=\sum_{n=-\infty}^\infty \left[|\kappa(\cF_0^n x)|
+|\kappa(\cF_0^n y)|-|\kappa(\cF_0^n v_1)|-|\kappa(\cF_0^n v_2)|\right]
$$
(to see that this series converges note that for $n\to+\infty$  $\cF_0^n x$ and $\cF^n v_2$ as well as
$\cF_0^n y$ and $\cF^n v_1$ become exponentially close
while for $n\to-\infty$  $\cF_0^n x$ and $\cF^n v_1$ as well as
$\cF_0^n y$ and $\cF^n v_2$ become exponentially close). The proof of mixing of the special flow with roof function
$|\kappa|$ depends on the estimates on oscillations of the $\Delta_{| \kappa}(x,y)$ when $x$ and $y$ are close.
Since $\Delta_{|\kappa|}=\Delta_{|\kappa|+b}$ the argument of \cite{Ch07} works for roof function
$|\kappa|+ b$ as well.
\end{proof}

Now we can prove the following

\begin{lemma}
The function $f$ is minimal, i.e. $M(f)=\mathbb Z \times \mathbb R.$
\end{lemma}
\begin{proof}
We claim that if $M(f)$ is a proper subgroup of $\mathbb Z \times \mathbb R$ then
there exist numbers $\alpha, r$ and measurable functions $h: \mathcal M_0 \rightarrow \mathbb R$
and $g:\mathcal M_0\rightarrow \mathbb{Z}$ such that
\begin{equation}
\label{eq:app1}
 |\kappa|(x) = h(x) - h(\mathcal F_0 x) +r+\alpha g(x).
 \end{equation}
Consider first the case when $M(f)$ is
one-dimensional.
By Theorem 5.1 in \cite{SzV04}, $\psi$ is minimal, hence the projection of $M(f)$ to the first coordinate is $\mathbb{Z}.$
Therefore if $M(f)$ is
one dimensional, then
the projection of $M(f)$ to the second coordinate is a discrete subgroup. Let us denote it by
$\mathcal L =  \alpha \mathbb Z$.
Clearly, $S(\tilde f_2) = \mathcal L$ and $\tilde f_2 \sim |\kappa|$ proving \eqref{eq:app1} in this case.

Next, consider the case when $M(f)$ is a two dimensional discrete subgroup of $\mathbb Z \times \mathbb R.$
We claim that the generators of $M(f)$ can be chosen of the form
\begin{equation}
\label{NotZ2}
(0, \alpha) \text{ and } (1, \beta).
\end{equation}
Indeed let
$e_1=(m_1, \alpha_1)$ and $e_2=(m_2, \alpha_2)$ be arbitrary generators. If either $m_1$ or $m_2$
is 0 we are done. Otherwise $m_1$ and $m_2$ need to be coprime since otherwise the projection
of $M(f)$ to the first coordinate would be a proper subgroup of $\mathbb Z.$ Thus we can take
$e=m_2 e_1-m_1 e_2$ as one of the generators and it is of the form $(0, \alpha).$ So if $\tilde e=(\tilde m, \tilde b)$ is
a second generator, then because $\psi$ is minimal we must have $\tilde m=\pm 1$ and we can ensure + sign replacing
$\tilde e$ by $-\tilde e$ if necessary.  \eqref{NotZ2} tells us that for some measurable functions $h_1, h_2$ we have
\begin{equation}
\label{2DimCoB}
 (\psi, | \kappa|)(x)-(r_1, r_2)=m'(0, \alpha)+n'(1, \beta)+(h_1, h_2)(x)-(h_1, h_2)(\cF_0 x). 
\end{equation}
Taking the first component of \eqref{2DimCoB} we obtain
$$ n'=\psi(x)-r_1-h_1(x)+h_1(\cF_0 x). $$
Now the second component of \eqref{2DimCoB} takes form
\begin{equation}
\label{FirstCoB}
| \kappa(x) | -r-\beta\psi(x) =m'\alpha+\tH(x)-\tH(\cF_0 x)
\end{equation}
where $\tH=h_2-\beta h_1.$ Let $x=(q,v)$ and $\cF_0(x)=(q_1, v_1).$
Then \eqref{FirstCoB} for the original and the time reversed orbits read
$$ | \kappa(q, v) | -r-\beta\psi(q,v) =m'\alpha+\tH(q, v)-\tH(q_1, v_1) \text{ and}$$
$$ | \kappa(q, v) | -r+\beta\psi(q,v) =m''\alpha+\tH(q_1, -v_1)-\tH(q, -v). $$
Adding them together we get \eqref{eq:app1} with $h(q,v)=\frac{1}{2} [\tH(q,v)+\tH(q_1, -v_1)].$

We now show that \eqref{eq:app1} contradicts to Lemma \ref{lemma:ap2mixing}.
Let us define the subset
\[ \mathcal C_{\delta} = \{ (x,t), x \in \mathcal M_0, t \in [h(x)-\delta, h(x)+ \delta]\} \subset \Omega. \]
$\mathcal C_{\delta}$ is measurable since $h$ and $\kappa$ are measurable.
Observe that $h$ is only defined up to an additive constant in (\ref{eq:app1}). Clearly one can choose this constant in
such a way that for any $\delta>0$, $ \mathcal C_{\delta}$ has a positive $\mu_0 \times Leb$-measure.
Now choose some $(x, h(x)) \in \mathcal C_0$ and write
\[ \varsigma (x) = \min_{s>0} \{ \Phi_{| \kappa|}^s(x,t) \in \mathcal C_0 \}\]
Using (\ref{eq:app1}) we conclude
\[ \varsigma (x) = - n r + \alpha \sum_{i=1}^n g(\mathcal F_0^{i-1} x),\]
where $n = n(x)$ is the number of hits of the roof before time $\varsigma (x)$.
Let us choose a positive $\varepsilon$ such that
$r - \varepsilon$ is a rational multiple of $\alpha$. Let us
denote by $\mathcal L'$ the lattice generated by the numbers $r - \varepsilon$ and
$\alpha$ and write $b$ for the smallest positive element of $\mathcal L'$.
Using the canonical embedding of $\Omega$ to the phase space of $\Phi_{|\kappa| + \varepsilon}$,
we conclude that for any  $(x, h(x)) \in \mathcal C_0$,
the first return time to $\mathcal C_0$ with the dynamics $\Phi_{|\kappa| + \varepsilon}$
is in $\mathcal L'$. Thus, taking $\delta>0$ smaller than
$b/2$, we conclude that for every
$t>0$ with $dist(t, \mathcal L') > 2 \delta$,
\[(\mu_0 \times Leb) (\mathcal C_{\delta} \cap \Phi_{|\kappa| + \varepsilon}^{-t} \mathcal C_{\delta} )=0. \]
This contradicts Lemma \ref{lemma:ap2mixing}.
Thus $f$ is minimal.
\end{proof}

Now we apply Theorem \ref{thm:SzV}
to conclude that (\ref{eq:A1first})
holds uniformly for $x,y$ chosen from a compact set and
\begin{equation}
\label{KatzFla}
\bc_{\mathcal A}=(Counting \times Leb \times \mu_0)(\mathcal{A})= 
\int \int_0^{\kappa(x)} 1 dt d\mu_0(x) = \bar \kappa 
\end{equation}
where the second identity follows by time reversal.

\subsection{Proof of Lemma \ref{lemma:llt} (a) for standard pairs}
\label{sec:ap2}

In this subsection, we prove that 
\begin{equation}
\label{equ:appendix0}
 n \vartheta_n (\mathcal A) \rightarrow \bar \kappa \varphi_{\Sigma}(x,y)
\end{equation}
holds when the initial measure is some standard pair $\ell=(\gamma, \rho)$. 
For brevity, let us write $\vartheta_n^{\nu}$ for the distribution of
\[ (\lfloor X_{n}(q,v) - x\sqrt n \rfloor , F_{n}(q,v) - n \bar \kappa - y \sqrt{ n},
\mathcal F_0^{n} (q,v)) \]
when the initial measure is some $\nu$.\\

Fix a small $\varepsilon>0$. 
As it was proven in \cite{Y98}, there exists some set $\mathfrak R \subset \mathcal M_0$
such that
\begin{enumerate}
 \item [(T1)] $\mathfrak R$ is in the domain $\mathfrak Q$ bounded by two stable and two unstable manifolds ($W^s_1$,
$W^s_2$, $W^u_1$, $W^u_2$), 
\item [(T2)] $\mu_0 (\mathfrak R) > (1- \varepsilon) \mu_0 (\mathfrak Q)$
 \item [(T3)] for every $x \in \mathfrak R$, the local stable and unstable manifolds through $x$ exist
and both of them fully cross $\mathfrak Q$ (i.e. $W^u(x) \cap W^s_i \neq \emptyset$
and $W^s(x) \cap W^u_i \neq \emptyset$ hold for $i=1,2$).
 \item [(T4)] for every $x,y \in \mathfrak R$ there is a unique $z \in \mathfrak R$ with 
$z= W^u(x) \cap W^s(y)$.
 \item [(T5)] The diameter of $\mathfrak Q$ is small enough so that both the ratio of the density of $\mu_0$ at different
points of $\mathfrak Q$ and the Jacobian of the holonomy map is in the interval $[1-\varepsilon, 1+\varepsilon]$.
 \item [(T6)] $\mathfrak R$ satisfies all Young's axioms ((P1) - (P5) in \cite{Y98}, 
their precise formulation is not needed for our argument). 
\end{enumerate} 
Namely, it is shown in \cite{Y98} that one can construct 
$\mathfrak R$ and $\mathfrak Q$ so that (T1), (T3), (T4) and (T6) are satisfied and, moreover,  the diameter of $\mathfrak{Q}$
can be taken arbitrary small. It remains to take $\mathfrak{Q}$ so small that (T2) and (T5) hold.

Let us fix this set $\mathfrak R$. 
Following the notation of \cite{Ch06}, we write
\[ \mathfrak S = \cup_{x \in \mathfrak R} W^s(x).\]
Further, let us fix an unstable manifold $\gamma^*$ that fully crosses $\mathfrak R$ and write 
$\pi: {\mathfrak S} \rightarrow \gamma^*$ with $\pi(x)=y$ if $x \in \mathfrak S$ and $y \in \gamma^*$ 
lie on the same stable manifold.

We claim that with the notation $\nu^B(.) = \nu(.|B)$, we have
\begin{equation}
\label{eq:ap2expanding}
n \vartheta_n^{\pi_* \mu_0^{\mathfrak R}} (\mathcal A) \rightarrow \bar \kappa \varphi_{\Sigma} (x,y).
\end{equation}
Indeed, Theorem 4.1 in \cite{SzV04} (which is our Theorem \ref{thm:SzV})
is intrinsically proven for the so-called expanding Young tower, which is constructed
over $\mathfrak R$ by factorizing along the stable direction. Hence the measure for which Theorem 4.1 in \cite{SzV04}
is first obtained is $\pi_* \mu_0^{\mathfrak R}$, which combined with Section \ref{sec:ap1} gives (\ref{eq:ap2expanding}).

Now we have the following
\begin{lemma}
\label{lemma:ap2}
 For every $\bar \varepsilon >0$ there is some $\varepsilon >0$ such that if $\mathfrak R$ satisfies (T1)--(T6) 
then the following  statement is true.
For every standard pair $\ell'=(\gamma', \rho')$ with $\gamma'$ fully crossing $\mathfrak S$, the density $\frac{d \rho'}{d Leb_{\gamma'}}$ is in the interval
$[1- \bar \varepsilon, 1+ \bar \varepsilon]$ and we have
\begin{equation}
\label{eq:ap2stpair}
| n \vartheta_n^{(\rho')^{\mathfrak S}} (\mathcal A)  - 
\bar \kappa \varphi_{\Sigma} (x,y)  | < \bar \varepsilon 
\end{equation} 
for $n$ large enough.
\end{lemma}

\begin{proof}
In fact, the conditions (T2) and (T5) are imposed exactly in order to enable the argument below.

By definition the densities of standard pairs are uniformly H\"older continuous, thus for $\varepsilon >0$
small enough, 
$\frac{d \rho'}{d Leb_{\gamma'}}$ is in the interval
$[1- \bar \varepsilon, 1+ \bar \varepsilon]$.
Whence by choosing $\varepsilon>0$ small, and using the definition of $\mathfrak R$ 
one easily concludes that the measures $\pi_* \rho'|_{\mathfrak S}$
and $\pi_* \mu_0^{\mathfrak R}$ are close to each other in the sense that their Radon-Nikodym derivative
w.r.t each other are in the interval $[1- \bar \varepsilon, 1+ \bar \varepsilon]$. 
The H\"older continuity of $f$ and the fact that stable manifolds are exponentially contracted by 
$\mathcal F_0^n$ implies that we can choose a small $\varepsilon >0$ such that for any integer $N$, for any 
$x \in \mathfrak S$, and $y = \pi x$, we have
\[ | \sum_{i=1}^N f \circ \mathcal F^i(x) - \sum_{i=1}^N f \circ \mathcal F^i(y) | < \bar \varepsilon.\]
Thus enlarging $\mathcal A$ a little bit to $\mathcal A^{\bar \varepsilon}$,
where
$$ \mathcal A^{\bar \varepsilon} =
\{ (n,r,\omega) : \exists r', \omega', |r-r'| < \bar \varepsilon,
\dist( \omega , \omega') < \bar \varepsilon, (n,r',\omega') \in \mathcal A\},
$$
we have both
\begin{equation}
\label{eq:ap2standardpair}
n \vartheta_n^{(\rho')^{\mathfrak S}} (\mathcal A) 
< (1+ \bar \varepsilon) n \vartheta_n^{\pi_* \mu_0^{\mathfrak R}} (\mathcal A^{\bar \varepsilon})
\end{equation}
and 
\[ n \vartheta_n^{\pi_* \mu_0^{\mathfrak R}} (\mathcal A) < (1+ \bar \varepsilon) n \vartheta_n^{(\rho')^{\mathfrak S}} (\mathcal A^{\bar \varepsilon}). \]
Finally, observe that by construction the
$(Counting \times Leb \times \mu_0)$
-measure of
$\partial \mathcal A$ is $0$. The statement follows.
\end{proof}

Now we can prove the convergence of $n \vartheta_n(\mathcal A) = n \vartheta_n^{\ell} (\mathcal A)$.\\
The Appendix of \cite{Ch06} implies the existence of a function 
\[ \Upsilon = (\Upsilon_1, \Upsilon_2): \gamma \rightarrow (\mathbb N \cup {\infty}) \times \mathbb N
\]
such that
\begin{itemize}
 \item There exist universal constants $\varkappa$, $C$ and $\theta <1$ depending only on the geometry of the billiard such that
$$\rho (\omega: \Upsilon_1(\omega) > \varkappa |\log \length (l)| + N) < C \theta^N.$$ In particular, $\rho (\omega: \Upsilon_1(\omega) = \infty) = 0$
 \item For any $\omega \in \gamma$ with $\Upsilon_1(\omega) < \infty$, 
$\mathcal F^{\Upsilon_1(\omega)} (\omega)$ lies on the translational copy of $\mathfrak S$ in the cell 
$ \Upsilon_2(\omega)$. 
\item For any $\omega \in \gamma$ with $\Upsilon_1(\omega) < \infty$, let us write $\gamma'_{\omega} \subset \gamma$ for the smallest subcurve 
of $\gamma$ which contains $\omega$ and
$\mathcal F_0^{\Upsilon_1(\omega)} \gamma'_{\omega}$ fully crosses $\mathfrak S$.
Then for $\omega' \in \gamma'_{\omega}$, the equation $\Upsilon_1(\omega') = \Upsilon_1(\omega)$ holds if and only if $\mathcal F^{\Upsilon_1(\omega)} \omega'$ lies 
on the translational copy of $\mathfrak S$ in the cell 
$ \Upsilon_2(\omega)$. In this case, $\Upsilon_2(\omega') = \Upsilon_2(\omega)$.
\end{itemize}

The meaning of the function $\Upsilon$ is that for a point $\omega$, the first $n$ such that 
$\mathcal F_0^n \gamma$ fully crosses $\mathfrak S$ and $\mathcal F_0^n(\omega)$ lies on $\mathfrak S$ is $\Upsilon_1$. 
But when we apply $\mathcal F^{\Upsilon_1(\omega)}$ instead of 
$\mathcal F^{\Upsilon_1(\omega)}_0$, the point $\omega$ arrives at some cell $\Upsilon_2(\omega)$.
Also note that by construction $|\Upsilon_2(\omega)| < \kappa_{\max} \Upsilon_1(\omega)$ for every $\omega$.\\

Now pick a large $n$ and some standard pair $\ell = (\gamma, \rho)$ with $|\log \length (\ell)| < n^{1/4}$.
For any $\omega \in \gamma$ with 
\begin{equation}
\label{eq:ap2coupling}
 \Upsilon_1(\omega) < \varkappa n^{1/4} + n^{1/5}
\end{equation}
we want to apply Lemma \ref{lemma:ap2} to the measure
\[ \left( \mathcal F_0^{\Upsilon_1(\omega)} \right)_* \rho|_{ \{ \omega' \in \gamma'_{\omega} \text{ such that } \Upsilon_1(\omega')= \Upsilon_1(\omega) \}}.\]
More precisely, we need to adjust the parameters of Lemma \ref{lemma:ap2} a little bit. Namely, we replace $n, x$ and $y$ by 
\begin{eqnarray}
 &&n' = n - \Upsilon_1(\omega) \label{eq:ap2n'}\\
&& x' = \frac{x \sqrt n - \Upsilon_2(\omega)}{ \sqrt{n - \Upsilon_1(\omega)} } \label{eq:ap2x'} \\
&& y' = \frac{y \sqrt n - \sum_{i=0}^{\Upsilon_1(\omega) -1} |\kappa \circ \mathcal F_0^i (\omega)| - \bar \kappa \Upsilon_1(\omega) }{\sqrt{n - \Upsilon_1(\omega)} } \label{eq:ap2y'},
\end{eqnarray}
respectively. Note that by construction, $n \sim n'$ and the pairs $x, x'$ and $y, y'$ are close to each other 
when $n$ is big, uniformly in $\ell$ and in the choice of $\omega$ as long as (\ref{eq:ap2coupling}) is true. 
Also note that by the first property of $\Upsilon$, the set of $\omega$'s not satisfying (\ref{eq:ap2coupling})
has measure less than $C \theta^{n ^{ 1/5}}$, which is negligible. 
Thus we conclude that 
\begin{equation*}
| n \vartheta_n^{\ell} (\mathcal A)  - 
\bar \kappa \varphi_{\Sigma} (x,y)  | < \bar \varepsilon 
\end{equation*} 
if $n$ is large enough. Since $\bar \varepsilon$ was arbitrary, (\ref{equ:appendix0}) follows.
Finally, observe that all the estimations in this subsection are uniform for $x,y$ chosen from a compact set.
Thus the convergence in (\ref{equ:appendix0}) is uniform for $x,y$ chosen from a compact set.

\subsection{Proof of Lemma \ref{lemma:llt} (b)}
\label{sec:ap3}

Our argument is similar to the one in Section \ref{sec:ap2}, with the main difference of fixing 
a small enough $\varepsilon>0$ (and not letting $\varepsilon \rightarrow 0$ in the end).
We use the notation of Section \ref{sec:ap2}.

We apply a simplified version of Lemma \ref{lemma:ap2} (since we only need (\ref{eq:ap2standardpair})). Namely, by choosing
some fixed $\bar \varepsilon$, say $\bar \varepsilon = 1$, we have by (\ref{eq:ap2standardpair})
\[  \vartheta_n^{\rho'|_{\mathfrak S}} (\mathcal A) < 2  \vartheta_n^{\pi_* \mu_0^{\mathfrak R}} ( \mathcal A^{1}). \]
Next, the computation of Appendix A.1-A.4 in \cite{P09} implies that there exist some $C_1, C_2 < \infty$
depending only on the geometry of the billiard such that for every $x,y$ and $n$,
\begin{equation}
\label{eq:ap3expanding}
n \vartheta_n^{\pi_* \mu_0^{\mathfrak R}} (\mathcal A^1) < C_1 \varphi_{\Sigma} (x,y) + C_2 n^{-1/2}.
\end{equation}
Indeed, even though the computation of P\`ene is formulated for the function $\kappa$ instead of our $f$, 
her arguments are more general, since the computations are 
done on the expanding Young tower. Thus replacing her Lemma 11 by Lemma 4.1 of \cite{SzV04} and enlarging 
$\mathcal A^1$ to $ \mathcal B$ where 
$ \mathcal B=\cup_{p\in \mathcal{A}^1} W^s(p)$ (so that 
$\mathcal B$ contains entire stable manifolds 
as required by P\`ene), we obtain \eqref{eq:ap3expanding} with
some $C_1> (counting \times Leb \times \mu_0) (\mathcal B).$ 
So we have some constants $C'_1, C'_2 < \infty$ 
depending only on the geometry of the billiard such that for every $x,y$ and $n$,
\begin{equation*}
n \vartheta_n^{\rho'|_{\mathfrak S}} (\mathcal A) 
< C'_1 \varphi_{\Sigma} (x,y) + C'_2 n^{-1/2}.
\end{equation*}
Now using the same argument as in the end of Section \ref{sec:ap2}, we conclude that 
\begin{equation}
 \label{eq:ap3stpair1}
(n - \varkappa n^{1/4} - n^{1/5}) \vartheta_n^{l} (\mathcal A) < C'_1 \varphi_{\Sigma} (\tilde x, \tilde y) + C'_2 (n - \varkappa n^{1/4} - n^{1/5})^{-1/2} + C \theta^{n^{1/5}},
\end{equation}
where $\tilde x = \tilde x(x,n),\tilde y = \tilde y(y,n)$ 
are such that $\varphi_{\Sigma} (\tilde x, \tilde y)$ maximizes 
$\varphi_{\Sigma} (x', y')$ over all possible $x',y'$
we can get in \eqref{eq:ap2x'} and \eqref{eq:ap2y'} when using some $\omega$
satisfying (\ref{eq:ap2coupling}). 
Clearly, there is a finite constant $C''_2$ such that
\begin{equation}
\label{eq:ap3stpair2}
 C'_2 (n - \varkappa n^{1/4} - n^{1/5})^{-1/2} + C \theta^{n^{1/5}} < C''_2 n^{-1/2}. 
\end{equation}
It is also not hard to deduce from the formulas (\ref{eq:ap2n'}), (\ref{eq:ap2x'}) and (\ref{eq:ap2y'})
that there is a $C$ depending only on the geometry of the billiard such that
\[ |x-\tilde x| < C (|x| n^{-3/4} + n^{-1/4}), \quad
|y-\tilde y| < C (|y| n^{-3/4} + n^{-1/4}) \]
and hence there exists a constant $N$ depending only on the geometry of the billiard such that
for every $n > N$,
\begin{equation}
\label{eq:ap3stpair3}
 \| (x,y) - (\tilde x, \tilde y) \| <  \frac{\| (x,y)\| }{3} +1.
\end{equation}
Note that the isocontours of the function $(x,y) \mapsto \varphi_{\Sigma} (x,y)$ are ellipsoids centered
at the origin with ratio of axes $\sqrt \lambda_1 : \sqrt \lambda_2$, where $\lambda_1 > \lambda_2>0$ are the 
eigenvalues of $\Sigma$. Let $R=||(x,y)||.$ If $R>6$ then
$R/3 +1 < R/2$ and considering two ellipsoids such that the major
axis of the smaller one is $R/2$, the minor axis of the bigger one is $R$, and both are isocontours, we conclude that
there are constants $C, C''_1$ depending only on the geometry of the billiard (e.g. $C$ can be $4 \lambda_1 / \lambda_2$) such that 
\begin{equation}
\label{eq:ap3stpair4}
  \varphi_{\Sigma} (\tilde x, \tilde y) < C''_1 \varphi_{C \Sigma } (x,y),
\end{equation}
provided that the vectors $(x,y), (\tilde x, \tilde y)$ satisfy (\ref{eq:ap3stpair3}) and $\| (x,y)\| > 6 .$
Clearly the restriction $||(x,y)||>6$ can be discarded by taking a bigger $C''_1$.
Now substituting   \eqref{eq:ap3stpair4} into \eqref{eq:ap3stpair1} and using \eqref{eq:ap3stpair2}
we conclude that there are constants $C, C'''_1, C''_2$ and $N$ depending only on the geometry of the billiard 
such that for every $x,y$ and $n > N$, we have
\[ (n - \varkappa n^{1/4} - n^{1/5}) \vartheta_n^{l} (\mathcal A) < C'''_1 \varphi_{C \Sigma} (x,y) + C''_2 {n}^{-1/2}. \]
Thus there exist constants $C_{1,final}, C_{2,final}$ 
depending only on the geometry of the billiard such that for every $x,y$ and for every $n$,
\[ n \vartheta_n^{l} (\mathcal A) < C_{1,final} \varphi_{C \Sigma} (x,y) + C_{2,final} {n}^{-1/2}. \]

\section{Properties of the Brownian Meander.}
\label{AppMeander}
\begin{proof}[Proof of Lemma \ref{lemma:meanderprop}]
Let us write $B^a_{\hat \sigma}(t)$ for a Brownian motion of variance $\hat \sigma^2$ with $B^a_{\hat \sigma} (0) = a$ and denote
\[ M^a_{\hat \sigma}(t)= \max_{s \in [0,t]} B_{\hat \sigma} (s) \text{ and }
m^a_{\hat \sigma}(t)= \min_{s \in [0,t]} B_{\hat \sigma} (s). \]
Let us compute $\phi_{\hat \sigma}(x,y)$ by conditioning on the value of $\mathfrak X_{\hat \sigma} (1-\delta_t)$ (call it $z$).
By the self-similar property of the Brownian motion,
\begin{eqnarray*}
\phi_{\hat \sigma}(x,y) &=&
\int_{0}^y \sqrt{1- \delta_t} \phi_{\hat \sigma}  \left( \frac{z}{\sqrt{1- \delta_t}}, \frac{y}{\sqrt{1- \delta_t}} \right)\\
&&\lim_{\Delta \rightarrow 0} \frac{1}{\Delta}P( B^z_{\hat \sigma}(\delta_t) \in [x,x+ \Delta], M^z_{\hat \sigma}(\delta_t)< y
|  m^a_{\hat \sigma}(\delta_t)> 0 ) dz 
\end{eqnarray*}
We can approximate this expression by 
\begin{enumerate}
 \item Substituting the second line by $\frac{1}{\Delta} \lim_{\Delta \rightarrow 0} P( B^z_{\hat \sigma}(\delta_t) \in [x,x+ \Delta])$,
which is clearly good approximation if $\delta_t$ is small and $z \in [\delta, y - \delta]$ for some fixed $\delta$ 
(the contribution of other $z$'s is small because of the first line)
 \item Replacing the integral by a sum over $\delta_s$. \qedhere
\end{enumerate}
\end{proof}

\begin{proof}[Proof of Lemma \ref{LmLimBox}]
We have
\begin{equation}
\label{NormInt}
\int_{\sqrt{\delta}}^{1/\sqrt{\delta}}
 \frac{|2k + x|}{\sqrt{2 \pi} \hat \sigma}\exp \left( -\frac{(2k+x)^2s^2}{2 \hat \sigma^2} \right)ds 
\end{equation}
$$  = \Phi \left( \frac{|2k+x|}{\hat \sigma \sqrt{\delta}} \right) -  \Phi \left( \frac{|2k+x| \sqrt \delta}{\hat \sigma} \right) $$
where $\Phi$ is the cumulative distribution function of the standard
Gaussian random variable. Computing the integral for $k=0$, we obtain
\begin{eqnarray} 
\label{eq:Ix}
I_x 
 =
\frac{\hat c_1 (\mathcal G) \sqrt{2 \pi}}{\hat \sigma} + 
\lim_{\delta\rightarrow 0}
\frac{2 \hat c_1 (\mathcal G) \sqrt{2 \pi}}{\hat \sigma}
& \sum_{k=1}^{\infty}  &
\Big[
\Phi \left( \frac{(2k+x)}{\hat \sigma \sqrt{\delta}} \right) 
-  \Phi \left( \frac{(2k-x)}{\hat \sigma \sqrt{\delta}} \right) \nonumber\\
& +& \Phi \left( \frac{(2k-x) \sqrt \delta}{\hat \sigma} \right) 
-  \Phi \left( \frac{(2k+x) \sqrt \delta}{\hat \sigma} \right)
\Big]. 
\end{eqnarray}
It is clear that
\[ 
\lim_{\delta\rightarrow 0}
\sum_{k=1}^{\infty}
\Big[
\Phi \left( \frac{(2k+x)}{\hat \sigma \sqrt{\delta}} \right) 
-  \Phi \left( \frac{(2k-x)}{\hat \sigma \sqrt{\delta}} \right)
\Big] =0.
\]
Let us write
\begin{equation}
\label{Dkxdelta}
d_{k,x,\delta} = 
\Phi \left( \frac{(2k-x) \sqrt \delta}{\hat \sigma} \right) 
-  \Phi \left( \frac{(2k+x) \sqrt \delta}{\hat \sigma} \right).
\end{equation}
Then
\[
I_x = \frac{\hat c_1 (\mathcal G) \sqrt{2 \pi}}{\hat \sigma} + 
\frac{2 \hat c_1 (\mathcal G) \sqrt{2 \pi}}{\hat \sigma}
\lim_{M\rightarrow \infty} \lim_{\delta \rightarrow 0}
\Big\{
\sum_{k=1}^{\lfloor \frac{M}{\sqrt \delta} \rfloor} d_{k,x,\delta}
+ \sum_{k={ \lfloor \frac{M}{\sqrt \delta} \rfloor  }}^{\infty} d_{k,x,\delta}
\Big\}.
\]
Denote the above sums by $S_{1,M,\delta}$ and $S_{2,M,\delta}$. 
Since $x<1$, it is easy to see that for $\delta<1$
\[ |S_{2,M,\delta}| < 1- \Phi \left( \frac{(2M/\sqrt \delta -x ) \sqrt \delta}{\hat \sigma} \right).\]
Thus
\[ \lim_{M \rightarrow \infty} \lim_{\delta \rightarrow 0}
S_{2,M,\delta} = 0.\]
Next, 
\[ d_{k,x,\delta} = -\frac{2x\sqrt \delta }{\hat \sigma} \varphi \left( \frac{(2k-x) \sqrt \delta}{\hat \sigma} \right)
+ o(\sqrt \delta)\]
as $\delta \rightarrow 0$ uniformly for $k < 
M/\sqrt \delta.$ Thus using the convergence of Riemannian sums, we 
conclude
\[ \lim_{\delta \rightarrow 0} S_{1,M,\delta} = -\frac{2x}{\hat \sigma} 
\int_{0}^M \varphi \left( \frac{2y}{\hat \sigma} \right)dy.\]
Consequently,
\[\lim_{M \rightarrow \infty}
\lim_{\delta \rightarrow 0} S_{1,M,\delta} = -x/2.\]
This completes the proof.
\end{proof}

\begin{proof}[Proof of Lemma \ref{LmBoundaryLayer}]
Similarly to \eqref{NormInt} we get that
$ u(t,x)=\sum_k \sgn(x+2k) v(t, x, k)$ with
$$ v(t,x, k)=1-\Phi\left(\frac{x+2k}{\hat\sigma \sqrt{t}}\right) .$$
An elementary computation shows that 
$$ v'_t = \frac12 \frac{x+2k}{\hat \sigma t^{3/2}} 
\varphi \left( \frac{x+2k}{\hat \sigma \sqrt t} \right) \text{ and } 
v''_{xx} =-\frac{1}{\hat \sigma \sqrt t} 
\varphi \left( \frac{x+2k}{\hat \sigma \sqrt t} \right)
\left( - \frac{x+2k}{\hat \sigma^2t}\right).
$$
Thus $u(t,x)$ satisfies our heat equation. 
In order to check the boundary conditions, we use the same argument as in the proof of Lemma \ref{LmLimBox}.
Namely, recalling \eqref{Dkxdelta} it is clear that
$$ \lim_{x \searrow 0} \sum_{k=1}^{\infty} d_{k,x,1/t} =0. $$
so the main contribution to $u$ comes from $k=0$ giving 
$$ \lim_{x \searrow 0} u(t,x) = 
\frac{2 \lambda_0 \hat c_1 (\cG) }{\hat \sigma^2} \lim_{x \searrow 0}
\int_{1/ \sqrt t}^{\infty} x \exp \left( - \frac{x^2s^2}{2 \hat \sigma^2} \right)
ds = \frac{ \lambda_0 \hat c_1 (\cG) \sqrt{2 \pi}}{\hat \sigma} =f_0.
$$
Likewise $ \lim_{x \nearrow 1} u(t,x) = 0.$
\end{proof}

\section*{Acknowledgement} P.N. thanks Lai-Sang Young for illuminating discussions.
Research of D.D. was partially supported by the NSF. Both authors thank Bernoulli
Center where a part of the paper was written for good working conditions.

\end{document}